\documentclass{amsproc}

\theoremstyle{definition}

\theoremstyle{remark}

\numberwithin{equation}{section}

\usepackage[all]{xy}

\def\E{\ifmmode{\mathbb E}\else{$\mathbb E$}\fi} 
\def\N{\ifmmode{\mathbb N}\else{$\mathbb N$}\fi} 
\def\R{\ifmmode{\mathbb R}\else{$\mathbb R$}\fi} 
\def\Q{\ifmmode{\mathbb Q}\else{$\mathbb Q$}\fi} 
\def\C{\ifmmode{\mathbb C}\else{$\mathbb C$}\fi} 
\def\H{\ifmmode{\mathbb H}\else{$\mathbb H$}\fi} 
\def\Z{\ifmmode{\mathbb Z}\else{$\mathbb Z$}\fi} 
\def\P{\ifmmode{\mathbb P}\else{$\mathbb P$}\fi} 
\def\T{\ifmmode{\mathbb T}\else{$\mathbb T$}\fi} 
\def\SS{\ifmmode{\mathbb S}\else{$\mathbb S$}\fi} 
\def\DD{\ifmmode{\mathbb D}\else{$\mathbb D$}\fi} 
\def\dudtau{{\frac{\del u}{\del \tau}}}
\def\dudt{{\frac{\del u}{\del t}}}
\def\delbar{{\overline \partial}}

\newcommand{\del}{\partial}
\def\CA{{\mathcal A}}

\def\CE{{\mathcal E}}
\def\CF{{\mathcal F}}

\def\CL{{\mathcal L}}
\def\CM{{\mathcal M}}

\newcommand{\ben}{\begin{enumerate}}
\newcommand{\een}{\end{enumerate}}
\newcommand{\be}{\begin{equation}}
\newcommand{\ee}{\end{equation}}
\newcommand{\bea}{\begin{eqnarray}}
\newcommand{\eea}{\end{eqnarray}}
\newcommand{\beastar}{\begin{eqnarray*}}
\newcommand{\eeastar}{\end{eqnarray*}}
\newcommand{\bc}{\begin{center}}
\newcommand{\ec}{\end{center}}

\theoremstyle{theorem}
\newtheorem{thm}{Theorem}[section]
\newtheorem{cor}[thm]{Corollary}
\newtheorem{lem}[thm]{Lemma}
\newtheorem{prop}[thm]{Proposition}

\theoremstyle{definition}
\newtheorem{defn}[thm]{Definition}
\newtheorem{rem}[thm]{Remark}
\newtheorem{ques}[thm]{Question}

\newtheorem{sit}[thm]{Situation}

\newtheorem*{thm*}{Theorem}

\numberwithin{equation}{section}

\begin{document}
\quad \vskip1.375truein

\title[Anchored Lagrangian submanifolds]
{Anchored Lagrangian submanifolds \\ and their Floer theory}

\author[K. Fukaya, Y.-G. Oh, H. Ohta, K.
Ono]{Kenji Fukaya, Yong-Geun Oh, Hiroshi Ohta, Kaoru Ono}
\thanks{KF is supported partially by JSPS Grant-in-Aid for Scientific Research
No.18104001 and Global COE Program G08, YO by US NSF grant \# 0503954, HO by JSPS Grant-in-Aid
for Scientific Research No.19340017, and KO by JSPS Grant-in-Aid for
Scientific Research, Nos. 18340014 and 21244002.}
\address{Department of Mathematics, Kyoto
University, Kyoto, Japan } \email{fukaya@math.kyoto-u.ac.jp}
\address{Department of Mathematics, University of
Wisconsin, Madison, WI, USA } \email{oh@math.wisc.edu}
\address{Graduate School of Mathematics,
Nagoya University, Nagoya, Japan } \email{ohta@math.nagoya-u.ac.jp}
\address{Department of Mathematics,
Hokkaido University, Sapporo, Japan }
\email{ono@math.sci.hokudai.ac.jp}

\begin{abstract}
We introduce the notion of (graded) anchored
Lagrangian submanifolds and use it
to study the filtration of Floer' s chain complex.
We then obtain an anchored version of Lagrangian
Floer homology and its (higher) product structures.
They are somewhat different from the more standard non-anchored version.
The anchored version discussed in this paper is more naturally related
to the variational picture of Lagrangian Floer theory and
so to the likes of spectral invariants.
We also discuss rationality of Lagrangian submanifold
and reduction of the coefficient ring of Lagrangian Floer cohomology
of thereof.
\end{abstract}
\date{July 1, 2009}
\maketitle
\tableofcontents
\section{Introduction}
\label{sec:intro}
Lagrangian Floer theory associates to each given pair of
Lagrangian submanifolds $L_0, L_1 \subset M$ a
group $HF(L_1,L_0)$, called the Floer cohomology group.
Floer cohomology group can be regarded as a
($\infty/2$-dimensional) homology group
of the space of paths $\Omega(L_0,L_1)$ joining
$L_0$ to $L_1$:
\begin{equation}\label{pathspace}
\Omega(L_0,L_1) = \{\ell: [0,1] \to M \mid \ell(0) \in L_0,\, \ell(1) \in L_1\}.
\end{equation}
Floer \cite{floer:intersect} used Morse theory to rigorously define this
cohomology group.
The exterior derivative of the `Morse function' Floer used is the \emph{action one-form} $\alpha$ defined
by
\begin{equation}\label{actiononeform}
\alpha(\ell)(\xi) = \int_0^1 \omega(\dot \ell(t),\xi(t)) \, dt
\end{equation}
for each tangent vector $\xi \in T_\ell\Omega(L_0,L_1)$.
\par
In general the one form $\alpha$ is closed but not necessarily exact.
So one needs to use Novikov's Morse theory \cite{Nov81} of closed one forms.
In order to take care of non-compactness of the moduli space of connecting orbits
which occurs from non-exactness of the closed one form involved,
Novikov uses a kind of formal power series ring, the so called Novikov ring
for his Morse theory of closed one forms.
\par
Floer, and later Hofer-Salamon \cite{HoSa95} and the fourth named author \cite{Ono96},
used a similar Novikov ring for Floer homology of
periodic Hamiltonian system.
The present authors also used a Novikov ring to study
Lagrangian Floer homology in \cite{fooo00}.
They however introduced a slightly different
ring which they call {\it universal Novikov ring}.
The same universal Novikov ring was used in \cite{Fuk02II} to
associate a filtered $A_{\infty}$ category
(Fukaya category) to a symplectic
manifold, which combine Floer cohomologies
of various pairs of Lagrangian submanifolds, together with
their (higher) product structures.
\par
In Section 5.1 \cite{fooo08}, the relationship between the
Floer cohomology over a (traditional) Novikov ring and the one over the
universal Novikov ring is discussed, which concerns pairs of
Lagrangian submanifolds. The discussion thereof involves a systematic
choice of associating base points on the connected components of $\Omega(L_0,L_1)$
when the pair $(L_0,L_1)$ varies. In this paper we extend this to the
cases of three or more Lagrangian submanifolds, which enter in
the product structure of Floer cohomology.
\par
We remark that the closed one form $\alpha$ above, determines a
single-valued function on an appropriate covering space of
$\Omega(L_0,L_1)$ {\it up to addition of a constant}. The choice of
this additive constant, which is closely related to the choice of a
base point, determines the filtration of Floer cohomology. When more
than two Lagrangian submanifolds are involved, to equip filtrations
of the Floer cohomologies `in a consistent way' for all pairs
$(L_0,L_1)$ is a somewhat nontrivial problem. The problem of finding
a systematic choice of the base point for the filtration shares some
similarity with the corresponding problems for the degree
(dimension) and for the orientation of the moduli space of
pseudo-holomorphic strips or polygons. (See Definition
\ref{abstractindex}.)
\par
For the purpose of systematically finding the base points of the path spaces,
we use the notion of \emph{anchored Lagrangian submanifold}.

\begin{defn}[Anchored Lagrangian submanifolds]\label{secanchored}
Fix a base point $y \in M$.
An \emph{anchor} of a Lagrangian submanifold $L \subset M$ to $y$ is a path
$\gamma : [0,1] \to M$ such that $\gamma(0) = y, \, \gamma(1) \in L$.
A pair $(L,\gamma)$ is called an anchored Lagrangian submanifold.
\end{defn}
\par

Roughly speaking, Floer cohomology group is a cohomology
group of a chain complex $CF(L_1,L_0)$ which is generated by
the set of intersections $L_0 \cap L_1$ and
whose boundary operator $\partial$ is defined by `counting' the
number of solutions $u: \R \times [0,1] \to M$ of
the (nonlinear) Cauchy-Riemann equation
\be\label{eq:CR}
\begin{cases} \dudtau + J \dudt = 0\\
u(\tau,0) \in L_0, \quad u(\tau,1) \in L_1.
\end{cases}
\ee
The moduli space of pseudo-holomorphic strips entering in this counting
problem is an appropriate compactification of the solution space.
\par
To properly defined the Floer cohomology group, we need to study:
\begin{enumerate}
\item {(Filtration):} a filtration of the Floer's chain complex $CF(L_1,L_0)$
\item {($\Z$-Grading):} a $\Z$-grading with respect to which $\partial$ has degree 1
\item {(Sign):} a sign on the generators which induces a
$\Z$-module (or at least $\Q$ vector space) structure on
$CF(L_1,L_0)$ with respect to which $\partial$ is a $\Z$-module
(resp. a $\Q$ vector space) homomorphism.
\end{enumerate}
In all of these structures, the relative version, i.e., the
`difference' between two generators $q, \, p \in L_0 \cap L_1$ is
canonically defined : for (1) it is the symplectic area, for (2) it
is nothing but the so called Maslov-Viterbo index
\cite{viterbo,floer:index} and for (3) it is based on the choice
of orientation of the determinant bundle $\det D_u\delbar \to
\CM(p,q;L_0,L_1)$ at a solution $u$ of \eqref{eq:CR}.
More precisely, the gluing formula for the indices shares a similar
behavior with the problems on (1) and (2).
(See Remarks \ref{connorbori} and \ref{fooo8-5}.)
Denote any of
these invariants associated to the Floer trajectory $u$ by
$I(q,p;u)$. The main problem to solve to provide these structures
then is to see if there exists some family of functions $I = I(q)$
\emph{independent of} the choice of $u \in \CM(p,q;L_0,L_1)$ such
that \be\label{eq:I(q)} I(q,p;u) = I(p) - I(q). \ee This is not
possible in general unless one puts various restrictions on the
triple $(L_0,L_1;M)$: for (1) exactness of $(M,\omega)$ and of
$(L_0,L_1)$, for (2) vanishing of $c_1$ of $(M,\omega)$ and of the
associated Maslov indices of $L_0, \, L_1$ and for (3) spinness of
the pair $L_0, \, L_1$ or (more generally \emph{relative spinness}
of the pair $(L_0,L_1)$). Under these restrictions respectively, it
has been well understood by now that such a choice is always
possible. See \cite{floer:intersect}, \cite{seidel:top} for (1),
\cite{floer:index} and \cite{fooo00} for (2) and \cite{fooo00} for
(3) respectively. (See also \cite{Fuk02II} where (1), (2) and (3)
are described in the setting of Fukaya category. There are some
technical errors and/or inconsistency with \cite{fooo08}, in the
description of \cite{Fuk02II}, which are corrected in this paper.)
\par
In this paper we define a filtered $A_{\infty}$ category on
each symplectic manifold, which is an anchored version
of Fukaya category. (See Theorem \ref{anchoredAinfty}.)
Its objects are anchored Lagrangian submanifolds $(L,\gamma)$
equipped with some extra data: (bounding cochain, spin structure and grading.)
The morphism is an (anchored version of) Floer's chain complex.
We however emphasize that the cohomology group $HF((L_1,\gamma_1),(L_0,\gamma_0))$
is {\it different} from usual Floer cohomology group
$HF(L_1,L_0)$ which is defined in \cite{fooo00,fooo06}. Namely the former is a component
of the latter where only one of the connected components of
$\Omega(L_0,L_1)$ is used for the construction.
The anchored versions of (higher) compositions $\frak m_k$ are also different from
the usual one.
The precise relationship between the anchored version and the non-anchored one
is rather complicate to describe.
\par
Necessity of studying this non-canonicality of filtration appears
in several situations: one is in the construction of Seidel's long
exact sequence as studied in \cite{oh:seidel} and the other is in the
study of \emph{Galois symmetry} in Floer homology \cite{fukaya:Galois}.
\par
Leaving the first problem to \cite{oh:seidel},
we will discuss the latter problem in Section \ref{sec:Galois} of this paper.
This involves the detailed discussion of the universal Novikov ring.
An element of the universal Novikov ring has the form
\be\label{eq:Novikov}
\sum_i a_i e^{\mu_i}T^{\lambda_i}
\ee
which is either a finite sum or an infinite sum with
$
\lambda_i \le \lambda_{i+1}, \, \lim_{i \to \infty} \lambda_i = + \infty.
$
Here $a_i $ is an element of a ground ring $R$ (for example $R=\Q$) and $\lambda_i$
are real numbers.
We consider the subring consisting of elements
(\ref{eq:Novikov}) such that $\lambda_i \in \Q$ in addition.
We denote it by $\Lambda_{0,nov}^{\text{\rm rat}}$.
We say that a Lagrangian submanifold $L$ of a symplectic manifold $M$
is \emph{rational} if the subgroup
$$
\Gamma_\omega(M,L) = \{\omega(\alpha) \mid \alpha \in \pi_2(M,L)\} \subset \R
$$
is discrete.
\par
Now we assume $[\omega] \in H^2(M;\Q)$. Then there exists
$m_{\text{\rm amb}} \in \Z_+$ and a complex line bundle
$\mathcal P$ with connection $\nabla$ such that
the curvature form of $(\mathcal P,\nabla)$ is
$2\pi \sqrt{-1} m_{\text{\rm amb}}\omega$.
We call $(\mathcal P,\nabla)$ the {\it pre-quantum bundle}.
A Lagrangian submanifold $L$ is called \emph{Bohr-Sommerfeld rational} or simply
{\it BS-rational} if holonomy group of the restriction of $(\mathcal P,\nabla)$
to $L$ is of finite order. (Such a Lagrangian submanifold is called `cyclic' in \cite{oh:cyclic} and just
`rational' in \cite{fukaya:Galois}. Since a Lagrangian submanifold
$L$ is called a Bohr-Sommerfeld orbit when the holonomy group is trivial,
the name `BS-rational' seems to be a more reasonable choice.)

\begin{thm}\label{GStheorem}
To each $(M,\omega)$ with $[\omega] \in H^2(M;\Q)$ with $c_1(M) =
0$, we can associate a filtered $A_{\infty}$ category with
$\Lambda_{nov}^{\text{\rm rat}}$ coefficients. Its object consists
of a system $(L,\mathcal L,sp,b,s,S_L)$ where $L$ is a BS-rational
Lagrangian submanifold of $M$ with its Maslov class $\mu_L = 0$,
$\mathcal L$ is a flat $U(1)$ bundle on $L$ with finite holonomy
group, $sp$ is a spin structure of $L$, $b$ is a bounding cochain,
$s$ is a $\Z$-grading, and $S_L$ is a rationalization of $L$. We
denote this category by ${\mathcal Fuk}_{\text{\rm rat}}(M,\omega)$.
\par
If $m_{\text{\rm amb}}[\omega] \in H^2(M;\Z)$, then
there exists a $m_{\text{\rm amb}}\widehat{\Z}$ action on this category which is compatible with the
$\widehat{\Z}$ action of $\Lambda_{0,nov}^{\text{\rm rat}}$ as
continuous Galois group.
\end{thm}
We will explain the notions appearing in the theorem in Section \ref{sec:Galois}.
In fact our attempt to further reduce to a smaller ring leads us to
considering a collection of Lagrangian submanifolds for which one can
associate an $A_\infty$ category over a Novikov ring
like $\Q[[T^{1/m}]][T^{-1}][e,e^{-1}]$.
\begin{thm}\label{thm:N-rational}
Let $(M,\omega)$ be rational and $(\mathcal P,\nabla)$ be the
pre-quantum line bundle of $m_{amb}\omega$. Then for each fixed $N
\in \Z_+$, there exists a filtered $A_\infty$ category ${\mathcal
Fuk}_{N}(M,\omega)$ over the ring $\C[[T^{1/N}]][T^{-1}][e,e^{-1}]$:
\begin{enumerate}
\item its objects are $(L, \CL, sp, b,S_L)$ where $L$ is a $N$
BS-rational Lagrangian submanifold, $\CL$ is a flat
complex line bundle with its holonomy group $G(L,\nabla)$
in $\{\exp(2\pi k\sqrt{-1}/N) \mid k \in \Z\}$.
\item
The set of morphisms between two such objects is
$$
CF(L_1, \mathcal L_1,b_1,sp_1,S_{L_1}),(L_0, \mathcal L_0,b_0,sp_0,S_{L_0});\C[[T^{1/N}]][T^{-1}][e,e^{-1}]).
$$
\end{enumerate}
\end{thm}
There are also the anchored versions of Theorems \ref{GStheorem}, \ref{thm:N-rational}.
See Subsetion \ref{subsec:seredanchor}.
\par
The category could be empty for some $N$. For example, one necessary
condition for ${\mathcal Fuk}_{N}(M,\omega)$ to be non-empty is
that $N$ should be divided by $m_{\text{\rm amb}}$. This leads us to the notions of
\emph{$N$-rational} Lagrangian submanifolds: $L$ is called \emph{$N$-rational} if
$(\mathcal P^{\otimes N/m_{\text{\rm amb}}},\nabla^{\otimes N/m_{\text{\rm amb}}})|_L$ is trivial.
Then ${\mathcal Fuk}_{N}(M,\omega)$ is generated by $N$-rational
Lagrangian submanifolds for each fixed $N$. The following question seems to
be interesting to study
\begin{ques} Is ${\mathcal Fuk}_{N}(M,\omega)$ generated by a finite number
of objects? More specifically, is the number of the Hamiltonian
isotopy class of compact BS $N$-rational Lagrangian submanifolds
finite?
\end{ques}
It is shown in Section \ref{sec:Galois} that the system $({\mathcal
Fuk}_{N}(M,\omega); <)$ with respect to the partial order `$N < N'$
if and only if $N | N'$' forms an inductive system. By definition,
${\mathcal Fuk}_{\text{\rm rat}}(M,\omega)$ will be the
corresponding inductive limit.
\section{Novikov rings}
\label{subsec:novikov}
The following ring was introduced in [FOOO00] which
plays an important role in the rigorous formulation of
Lagrangian Floer theory.
\begin{defn}[Universal Novikov ring]
Let $R$ be a commutative ring with unit. (In many cases, we take $R=\Q$.)
We define
\begin{eqnarray}
\Lambda_{nov} & = & \left\{\sum_{i=1}^\infty a_i T^{\lambda_i}
e^{\mu_i/2} ~\Big|~ a_i \in R, \, \lambda_i \in \R, \, \lambda_i
\leq \lambda_{i+1}, \,
\lim_{i\to \infty}\lambda_i = \infty \right\}\\
\Lambda_{0,nov} & = &\left\{\sum_{i=1}^\infty a_i T^{\lambda_i}
e^{\mu_i/2}\in \Lambda_{nov} ~\Big|~ \lambda_i \geq 0 \right\}.
\end{eqnarray}
\end{defn}
There is a natural filtration on these rings provided by the multiplicative non-Archimedean
valuation defined by
\begin{equation}\label{eq:vlambda1}
v\left(\sum_{i=1}^\infty a_i T^{\lambda_i}e^{\mu_i/2} \right): =
\inf \left\{\lambda_i \mid a_i \ne 0\right\}.
\end{equation}
Here we assume $(\lambda_i,\mu_i) \ne (\lambda_j,\mu_j)$ for $i\ne j$.
\par
(\ref{eq:vlambda1}) is well-defined by the definition of the Novikov ring and
induces a filtration $F^\lambda\Lambda_{nov}: = v^{-1}([\lambda,
\infty))$ on $\Lambda_{nov}$. The function $e^{-v}: \Lambda_{nov}
\to \R_+$ also provides a natural non-Archimedean norm on
$\Lambda_{nov}$.
\par
Let $\overline C$ be a free $R$ module. We consider
$\overline C \otimes_R \Lambda_{nov}$
or $\overline C \otimes_R \Lambda_{0,nov}$.
We define valuation $v$ on it by
$v(\sum x_i \text{\bf e}_i) = \inf v(x_i)$, where $\text{\bf e}_i$
is the basis of $\overline C$. It defines a metric.
We take the completion and denote it by
$\overline C \,\widehat{\otimes}_R \,\Lambda_{nov}$
or $\overline C \,\widehat{\otimes}_R \,\Lambda_{0,nov}$, respectively.
\par
In the point of view of Novikov's Morse theory of closed one forms,
it is natural to use the version of Novikov ring that
is a completion of the group ring of an appropriate quotient group
of the fundamental group of $\Omega(L_0,L_1)$.
This is the point of view taken in many classical references of
various Floer theories. The universal Novikov ring
introduced above is slightly different from this Novikov ring.
In this paper we also use this more traditional Novikov ring,
whose definition is now in order.
\par
We consider the space of paths (\ref{pathspace}),
on which we are given the action one-form $\alpha$ (\ref{actiononeform}).
By definition
$$
\operatorname{Zero}(\alpha) = \{\widehat p: [0,1] \to M \mid p \in
L_0\cap L_1, \,\,\, \widehat p\equiv p \}.
$$
Note that $\Omega(L_0,L_1)$ is not connected but has countably many
connected components. We pick up a based path $\ell_{01} \in
\Omega(L_0,L_1)$ and consider the corresponding component
$\Omega(L_0,L_1;\ell_{01})$. We now review the definition of Novikov
covering we used in Section \ref{subsec:novikov}. Let $ g :
\widetilde{\Omega}(L_0,L_1;\ell_{01}) \to
\widetilde{\Omega}(L_0,L_1;\ell_{01}) $ be an element of deck
transformation group of the universal cover
$\widetilde{\Omega}(L_0,L_1;\ell_{01})$ of
${\Omega}(L_0,L_1;\ell_{01})$. It induces a map $ w : [0,1]^2 \to M
$ with $w(0,t) = \ell_{01}(t) = w(1,t)$, $w(s,0) \in L_0$, $w(s,1)
\in L_1$. (Namely $s \mapsto w(s,\cdot)$ represent the path
corresponding to $g$.) We put
\begin{equation}\label{defEw}
E(g) = \int_{[0,1]^2} w^*\omega.
\end{equation}
We also obtain a Lagrangian loop
$\alpha_{\lambda_{01};\lambda_{01}}$ defined on $\partial [0,1]^2$
by
\begin{equation}\label{condw3}
\aligned &\alpha_{\lambda_{01};\lambda_{01}}(0,t) =
\alpha_{\lambda_{01};\lambda_{01}}(1,t) =\lambda_{01}(t),
\\
&\alpha_{\lambda_{01};\lambda_{01}}(s,0) = T_{w(s,1)}L_0, \quad
\alpha_{\lambda_{01};\lambda_{01}}(s,1) = T_{w(s,1)}L_1.
\endaligned
\end{equation}
Here $\lambda_{01}$ is any path of Lagrangian subspaces along
$\ell_{01}$ with
$$
\lambda_{01}(0) = T_{\ell_{01}(0)}L_0, \quad
\lambda_{01}(1) = T_{\ell_{01}(1)}L_1.
$$
We denote by $\mu(g)$ be the Maslov index of this Lagrangian
loop.
See Section \ref{subsec:bundlepairs} for the definition of Maslov index of this Lagrangian
loop.
We remark that this index does
not depend on the choice of $\lambda_{01}$ and can be expressed as
the index of a bundle pair over the annulus independently of this
choice. (See \cite{fooo00}.)
\begin{defn}\label{novcov}
The Novikov covering is the covering space of
${\Omega}(L_0,L_1;\ell_{01})$ which corresponds to the kernel of the
homomorphism
$$
(E,\mu) : \pi_1({\Omega}(L_0,L_1;\ell_{01})) \to \R \times \Z.
$$
\end{defn}
Since $\Pi(L_0,L_1;\ell_{01})$ is the deck transformation group of
Novikov covering it follows that there exists an (injective) group
homomorphism
\begin{equation}\label{Eandeta}
(E,\mu) : \Pi(L_0,L_1;\ell_{01}) \to \R \times \Z.
\end{equation}
Let $\Lambda_{nov}$ be the field of fraction of $\Lambda_{0,nov}$.
$(E,\mu)$ induces a ring homomorphism
\begin{equation}\label{maptouniN}
\Lambda(L_0,L_1;\ell_{01}) \to \Lambda_{nov}
\end{equation}
by
$$
\sum_g c_g [g] \mapsto \sum c_g e^{\mu(g)/2} T^{E(g)}.
$$
\par
On $\widetilde\Omega(L_0,L_1;\ell_{01})$ we have a unique single
valued action functional $\CA$ such that
$$
d\CA = \pi^*\alpha, \quad \CA([\widetilde\ell_{01}]) = 0
$$
where $[\widetilde\ell_{01}]$ is a base point of $\widetilde \Omega(L_0,L_1;\ell_{01})$.
\par
We then denote by $\Pi(L_0,L_1;\ell_{01})$ the
group of deck transformations. We define the associated Novikov ring
$\Lambda (L_0,L_1;\ell_{01})$ as a completion of the group ring
$\Q[\Pi(L_0,L_1;\ell_{01})]$.
\begin{defn} $\Lambda_k (L_0,L_1;\ell_{01})$ denotes the set of
all (infinite) sums
$$\sum_{g\in \Pi(L_0,L_1;\ell_{01}) \atop \mu (g) = k} a_g
[g]$$ such that $a_g \in \Q$ and that for each $C \in \R$, the set
$$
\# \{ g \in \Pi(L_0,L_1;\ell_{01}) \mid E(g) \leq C, \,\,a_g \not = 0\}
< \infty.
$$
We put $\Lambda (L_0, L_1;\ell_{01}) = \bigoplus_k \Lambda_k (L_0,
L_1;\ell_{01})$.
\end{defn}
We call this graded ring the \emph{Novikov ring} of the pair
$(L_0,L_1)$ relative to the path $\ell_{01}$. Note that this ring
depends on the connected component of $\ell_{01}$.
\section{Anchors and abstract index}
\label{sec:gluinghomotopy}
In this paper we always assume that $L_0$ intersects $L_1$ transversely.
\par
Let $p, \, q \in
L_0 \cap L_1$. We denote by $\pi_2(p,q)=\pi_2(p,q;L_0,L_1)$ the set
of homotopy classes of smooth maps $u: [0,1] \times [0,1] \to M$
relative to the boundary
$$
u(0,t) \equiv p , \quad u(1,t) = q; \quad u(s,0) \in L_0, \quad
u(s,1) \in L_1
$$
and by $[u] \in \pi_2(p,q)$ the homotopy class of $u$ and by $B$ a
general element in $\pi_2(p,q)$. For given $B \in \pi_2(p,q)$, we
denote by ${Map}(p,q;B)$ the set of such $w$'s in class $B$. Each
element $B \in \pi_2(p,q)$ induces a map given by the obvious gluing
map $[p,w] \mapsto [q,w \# u]$ for $u \in Map(p,q;B)$. There is also
the natural gluing map \be\label{eq:pi2pqr} \pi_2(p,q) \times
\pi_2(q,r) \to \pi_2(p,r) \ee induced by the concatenation $(u_1,
u_2) \mapsto u_1\# u_2$. These `relative' homotopy classes are
canonically defined.
\par
On the other hand, if we have chosen a base path $\ell_{01} \in \Omega(L_0,L_1)$,
then we can define the set of path homotopy classes of the maps
$
w:[0,1]^2 \to M
$
satisfying the boundary condition
\begin{equation}\label{condw}
w(0,t) = \ell_{01}(t), \quad w(1,t)\equiv p, \quad w(s,0) \in L_0, \quad w(s,1) \in L_1.
\end{equation}
We denote the corresponding set of homotopy classes of the maps by
$\pi_2(\ell_{01};p)$. Then we have the obvious gluing map
\be\label{eq:pi2pell0}
\pi_2(\ell_{01};p) \times \pi_2(p,q) \to \pi_2(\ell_{01};q); (\alpha,B) \mapsto \alpha \# B.
\ee
\par
Now we would like to generalize this construction for a chain
$
\frak L = (L_0,\cdots, L_k)
$
of more than two Lagrangian submanifolds, i.e., with $k \geq 2$. (We
call such $\frak L$ the {\it Lagrangian chain} and $k+1$ the
\emph{length} of $\frak L$.)
\par
To realize this purpose, we
use the notion of \emph{anchors}
of Lagrangian submanifolds in this paper.

\begin{defn}\label{defn:anchored}
Fix a base point $y$ of ambient symplectic manifold $(M,\omega)$.
Let $L$ be a Lagrangian submanifold of $(M,\omega)$. We define an
\emph{anchor} of $L$ to $y$ is a path $\gamma :[0,1] \to M$ such that
$
\gamma(0) = y, \, \gamma(1) \in L.
$
We call a pair $(L,\gamma)$ an \emph{anchored} Lagrangian submanifold.
\par
A chain $\CE = ((L_0,\gamma_0),\cdots,(L_k,\gamma_k))$
is called an {\it anchored Lagrangian chain}.
$\frak L = (L_0,\cdots, L_k)$ is called
its underlying Lagrangian chain.
\end{defn}

It is easy to see that any homotopy class of path in
$\Omega(L,L')$ can be realized by a path that passes through the
given point $y$. Motivated by this observation, when we are given a
Lagrangian chain
$
(L_0, L_1, \cdots, L_k)
$
we also consider a chain of anchors $\gamma_i: [0,1] \to M$ of $L_i$ to $y$
for $i = 0, \cdots, k$. These anchors give a systematic choice of a base path
$\ell_{ij} \in \Omega(L_i,L_j)$ by concatenating $\gamma_i$ and
$\gamma_j$:
\begin{equation}\label{ellij}
\ell_{ij}(t) =
\begin{cases} \gamma_i(1-2t) &t\le1/2 \\
\gamma_j(2t-1) &t\ge 1/2.
\end{cases}
\end{equation}
The upshot of this construction is the following
overlapping property
\begin{equation}\label{eq:ellij}
\ell_{ij}(t) = \ell_{i\ell}(t) \quad \text{for } \, 0 \leq t
\leq \frac{1}{2}, \qquad
\ell_{ij}(t) = \ell_{\ell j}(t) \quad \text{for }\, \frac{1}{2}
\leq t \leq 1
\end{equation}
for all $j, \, \ell$.
\par
Let $(L_0,\cdots,L_k)$ be a Lagrangian chain and $p_{(i+1)i}
\in L_i \cap L_{i+1}$. ($p_{(k+1)k} = p_{0k}$ and $L_{k+1} = L_0$ as convention.)
We write $\vec p = (p_{10},\cdots,p_{k(k-1)})$.
Let $\chi_{i} = \exp(-2\pi i\sqrt{-1}/k)$.
We consider the set of homotopy class of maps
$v : D^2 \to M$ such that $v(\overline{\chi_{i+1} \chi_{i} }) \subset L_i$
and $v(\chi_i) = p_{i(i+1)}$.
We denote it by $\pi_2(\frak L;\vec p)$.
If $\CE$ is an anchored Lagrangian chain and $\frak L$ be its
underlying Lagrangian chain we write $\pi_2(\CE;\vec p)$
in place of $\pi_2(\frak L;\vec p)$ some times by abuse of notation.
\begin{defn}\label{classB} Let $\CE = \{(L_i,\gamma_i)\}_{0 \leq i \leq k}$ be
a chain of anchored Lagrangian submanifolds.
A homotopy class $B \in \pi_2(\CE;\vec p)$
is called \emph{admissible} to $\CE$ if it
can be obtained by a polygon that is a gluing of $k$ bounding
strips $w_{i(i+1)}^-: [0,1] \times [0,1] \to M$ satisfying
\begin{subequations}\label{wanch9or}
\begin{eqnarray}
w_{i(i+1)}^-(s,0) & \in & L_i, \quad w^-_{i(i+1)}(s,1) \in L_{i+1} \label{3.5form}\\
w_{i(i+1)}^-(0,t) & = & p_{(i+1)i}. \label{3.6form} \\
w_{i(i+1)}^-(1,t) & = & \begin{cases}
\gamma_i(1-2t) \quad & 0 \leq t \leq \frac{1}{2} \\
\gamma_{i+1}(2t-1) \quad & \frac{1}{2} \leq t \leq 1
\end{cases}
\end{eqnarray}
\end{subequations}
When this is the case, we denote the homotopy class $B$ as
$$
B = [w^-_{01}]\#[w^-_{12}] \# \cdots \# [w^-_{k0}]
$$
and the set of admissible homotopy classes by $\pi_2^{ad}(\CE;\vec p)$.
\end{defn}
We note that
not all homotopy classes in $\pi_2(\CE;\vec p)$ is admissible
for a given anchored Lagrangian chain.
(See however Lemma \ref{existanchor2}.)
\begin{defn}
Let $(L_i,\gamma_i)$, $i=0,1$ be anchored Lagrangian submanifolds.
We say $p \in L_0 \cap L_1$ is {\it admissible} (with respect to the
pair $((L_0,\gamma_0),(L_1,\gamma_1))$) if
there exists
$w = w_{01}$ satisfying $(\ref{3.5form})$ for $i=0$ and
$(\ref{3.6form})$ for $i=0$, $p_{10} = p$.
\end{defn}
Note $p$ is admissible if and only if $\pi_2(\ell_{01};p)$ is nonempty.
(Here $\ell_{01}$ is as in (\ref{ellij}).)
Let us go back to the case $k=1$. First note that we have:
\begin{equation}\label{k2pi2compare}
\pi_2(p,q) = \pi_2(\frak L;\vec p)
\end{equation}
where $\frak L = (L_0,L_1)$, $\vec p = (p,q)$ and
the left hand side is as in the beginning of this section.
\begin{lem}\label{admissibilityfor2}
Let $k=1$ and $\CE = ((L_0,\gamma_0),(L_1,\gamma_1))$.
Then
$\pi_2^{ad}(\CE,(p,q))= \pi_2(\CE,(p,q))$ if $p,q$ are admissible.
Otherwise $\pi_2^{ad}(\CE,(p,q))$ is empty.
\end{lem}
The proof is easy and omitted.
\begin{lem}\label{existanchor1}
Let $L_0,L_1$ be a pair of Lagrangian submanifold and $p \in L_0 \cap L_1$.
Then for each given anchor $\gamma_0$ of $L_0$ there exists an anchor
$\gamma_1$ of $L_1$ such that $p$ is admissible
with respect to the
pair $((L_0,\gamma_0),(L_1,\gamma_1))$.
\end{lem}
The proof is easy and omitted.
\par
The proof of the following two lemmas are also easy and so is omitted.
\begin{lem}\label{existanchor2}
Let $\frak L$ be a Lagrangian chain and $B \in \pi_2(\frak L;\vec p)$.
Then there exist anchors $\gamma_i$ of $L_i$ $(i=0,\cdots,k)$ such that
$B$ is admissible with respect to $\CE$, where
$$
\CE = ((L_0,\gamma_0),\cdots,(L_k,\gamma_k)).
$$
\end{lem}
The anchors in Lemmas \ref{existanchor1}, \ref{existanchor2} are not
necessarily unique (even up to homotopy).
It is rather complicated to describe how many there are.
(See Section \ref{sec:relation} for some illustration.)
The following definition can be used to study the gluing formulas of
symplectic areas and Maslov indices of pseudo-holomorphic polygons
that enter in the construction of the anchored version of Fukaya category.
\begin{defn}\label{abstractindex}
Let $R$ be a module. We say a collection of maps
$$
I = \{I_k : \pi_2^{ad}(\CE;\vec p) \to R\}_{k=1}^\infty
$$
an {\it abstract index} over the collection of
anchored Lagrangian chains $\CE$, if they satisfy the following gluing
rule: whenever the gluing is defined, we have
$$
I_{k+1}([w^-_{01}]\# \cdots \# [w^-_{(k-1)k}] \# [w^-_{k0}]) = \sum_{i=0}^k
I_1([w^-_{i(i+1)}]).
$$
\end{defn}

In subsection
\ref{subsec:seredanchor}, we will use another abstract index, a
\emph{normalized symplectic area} over the class of BS-rational
Lagrangian submanifolds with $R = \Q$ or with $R = \frac{1}{N}\cdot
\Z$ for integers $N$.
\section{Anchors, action functional and action spectrum}
\label{sec:pointed}
For given two anchors $\gamma, \, \gamma'$ homotopic to each other,
we denote by $\pi_2(\gamma,\gamma';L)$ the set of homotopy classes of the
maps $w: [0,1]^2 \to M$ satisfying
$$
 w(0,t) = \gamma(t), \, w(1,t) = \gamma'(t), \, w(s,0)\equiv y,\,
 \text{and } \,
w(s,1) \in L.
$$
For any such map $w$, we define \be\label{eq:Agg'L}
a_{(\gamma,\gamma';L)}(w) = \int w^*\omega. \ee It is immediate to
check that this function pushes down to $\pi_2(\gamma,\gamma':L)$
which we again denote by $a_{(\gamma,\gamma';L)}$.

We denote by $G(\gamma,\gamma':L) \subset \R$ the image of
$a_{(\gamma,\gamma';L)}$. The following is easy to check whose proof
we omit.
Let $\overline{\gamma}*\gamma'$ be an element $\Omega(L,L;M))$
obtained by concatenating $\overline{\gamma}$
(where $\overline{\gamma}(t) = \gamma(1-t)$) and $\gamma'$
in the same way as (\ref{ellij}), and
$\Omega_{\overline{\gamma}*\gamma'}(L,L;M))$ the connected component
of $\Omega(L,L;M))$ containing it.
\begin{lem} $\pi_2(\gamma,\gamma':L)$ is a principal homogeneous space of
$\pi_1(\Omega_{\overline{\gamma} *\gamma'}(L,L;M))$.
and so $G(\gamma,\gamma':L)$ is a principal
homogeneous space of the group
$$
\{ \omega(C) \mid C \in \pi_1(\Omega_{\overline\gamma *\gamma'}(L,L;M)) \}.
$$
\end{lem}
The action functional $\CA = \CA_{(\gamma_0,\gamma_1;L)}: \widetilde
\Omega(L_0,L_1;\ell_{01}) \to \R$ is defined by
$$
\CA([\ell,w]) = \int w^* w.
$$
\par
Note an element of $\widetilde{\Omega}(L_0,L_1;\ell_{01})$
is identified with a pair $[\ell,w]$ where $\ell \in {\Omega}(L_0,L_1;\ell_{01})$
and $w : [0,1]^2 \to M$ satisfies
\begin{equation}\label{coverelement}
w(0,t) = \ell_{01}(t), \quad w(1,t)\equiv \ell(t), \quad w(s,0) \in L_0, \quad w(s,1) \in L_1.
\end{equation}
We identify $[\ell,w]$ with $[\ell,w']$ if
\begin{equation}\label{coveridentify}
\int (\overline{w}'\#w)^*\omega = 0, \quad \mu(\overline{w}'\#w) = 0.
\end{equation}
Here $\overline{w}'(s,t) = w'(1-s,t)$ and $\mu$ is an appropriate Maslov index.
(See (\ref{condw3}) and Definition \ref{novcov}.)
\par
We now study dependence of the action functional $
\CA_{(L_0,\gamma_0),(L_1,\gamma_1)} $ on their anchors. Let $\gamma_0,
\, \gamma_0'$ and $\gamma_1, \, \gamma_1'$ be two anchors of $L_0$
and $L_1$ respectively. It defines $\ell_{01}$ and $\ell'_{01}$ by (\ref{ellij}).
We assume that there exist
paths $w_0$, $w_1$ connecting them respectively.
Then $\overline w_0 \# w_1$ induces a diffeomorphism
$$
\Phi_{\overline w_0 \# w_1} : \widetilde \Omega(L_0,L_1;\ell_{01})
\to \widetilde \Omega(L_0,L_1;\ell'_{01})
$$
defined by
\be\label{eq:Phiw0w1}
\Phi_{\overline w_0 \# w_1}([\ell,u]) = [\ell, (\overline w_0 \# w_1) \# u].
\ee
For the clarity of notations, we will use $\#$ for two dimensional concatenations and
by $*$ for one dimensional ones.

\begin{prop}\label{prop:CAgamma} Let $\gamma_i, \, \gamma_i'$ and $w_i$ for $i = 0, \, 1$ be as above.
Consider $[\ell,u] \in \widetilde \Omega(L_0,L_1; \ell_{01})$. Then we have
$$
\CA_{(L_0,\gamma_0),(L_1,\gamma_1)} - \Phi_{\overline w_0 \# w_1}^*
\CA_{(L_0,\gamma_0'),(L_1,\gamma_1')} \equiv \omega([\overline w_0 \# w_1]).
$$
\end{prop}
\begin{proof} Obvious from the definition.
\end{proof}
Now we define:
\begin{defn}[Action spectrum]
Denote by $\operatorname{Spec}((L_0,\gamma_0),(L_1,\gamma_1))$ the
set of critical values of $\CA_{(L_0,\gamma_0),(L_1,\gamma_1)}$ and
call the \emph{action spectrum} of the pair $(L_0,\gamma_0)$,
$(L_1,\gamma_1)$.
\end{defn}
An immediate corollary of Proposition \ref{prop:CAgamma} and this definition is
the following
\begin{cor}\label{cor:spectrum}
We assume that $\gamma_i$ is homotopic to $\gamma_i'$ for $i=0,\,1$.
Then there exists a real constant
$c = c((L_0,\gamma_0),(L_1,\gamma_1);(L_0,\gamma'_0),(L_1,\gamma'_1))$ depending on the pair
$(L_0,\gamma_0),\, (L_1,\gamma_1)$ such that
$$
\operatorname{Spec}((L_0,\gamma_0),(L_1,\gamma_1)) =
\operatorname{Spec}((L_0,\gamma_0'),(L_1,\gamma_1')) +
c
$$
as a subset of $\R$.
\end{cor}
\begin{proof} Let $\gamma_i, \, \gamma_i'$ and $w_i$ for $i = 0, \, 1$ be as above.
By Proposition \ref{prop:CAgamma}, we have
$$
\operatorname{Crit} \CA_{(L_0,\gamma_0'),(L_1,\gamma_1')} + \omega([\overline w_0 \# w_1])
= \operatorname{Crit} \CA_{(L_0,\gamma_0),(L_1,\gamma_1)}
$$
for any choice of $w_0, \, w_1$ joining $\gamma_0, \, \gamma_0'$ and
$\gamma_1, \, \gamma_1'$ respectively.

Just take $c = \omega([\overline w_0 \# w_1])$. This finishes the proof.
\end{proof}
Next we consider the Lagrangian chains with 3 or more elements in
them. When we are given an
anchored Lagrangian chain
$$
\CE = ((L_0,\gamma_0), (L_1,\gamma_1), \cdots, (L_k,\gamma_k))
$$
these anchors give a systematic choice of a base path
$\ell_{ij} \in \Omega(L_i,L_j)$ by concatenating $\gamma_i$ and
$\gamma_j$ as in (\ref{ellij}).
Inside the collection of anchored Lagrangian submanifolds
$(L,\gamma)$ we are given a coherent system of single valued action functionals
$$
\CA : \widetilde{\Omega_0}(L_i,L_j;\ell_{ij}) \to \R.
$$
We will use the action functional associated to $\ell_{ij}$ to define an
energy level on the critical point set
$
\CA:\operatorname{Crit} \CA \to \R.
$
By the overlapping property \eqref{eq:ellij},
the following proposition is immediate whose proof we omit.
\begin{prop}
Denote by $\CE$ an anchored Lagrangian chain. Consider the map
$I_{\omega,k}: \pi_2(\vec p; \CE) \to \R$ defined by the symplectic area
$I_{\omega,k}(\alpha) = \omega(\alpha)$
for $k = 1, \cdots,$. Then the collection denoted by $I_\omega = \{I_{\omega,k}\}_{k=1}^\infty$
defines an abstract index of anchored Lagrangian chains.
\end{prop}
\section{Grading and filtration}
\label{sec:grading}
A familiar description of generators of Floer chain module as
the set of equivalence classes $[p,w]$ in the Novikov covering space
is useful as far as the study of \emph{filtration} on the Floer
complex is concerned. However for the study of \emph{grading} and
\emph{signs} on the Floer complex, we have to have additional
structures on the Floer chain module which requires some geometric
condition on the Lagrangian side, e.g., spin structure or graded
structure. There has been a few different approach to how one
incorporates these additional structures. In this section, we
describe them by using anchors.
\subsection{Maslov index in Lagrangian Grassmannian}
\label{subsec:bundlepairs}
In this subsection, we review the definition of Maslov index in
Lagrangian Grassmannian. The Lagrangian Grassmannian $Lag(S,\omega)$
of a symplectic vector space $(S,\omega)$ is defined to be
$$
Lag(S,\omega) =\{V \mid V \text{ is a Lagrangian subspace of
$(S,\omega)$} \}.
$$
When we equip $S$ a compatible complex structure $J$ and
define $U(S)$ to be the group of unitary transformations of $S$, any
$V_0,V_1 \subset Lag(S,\omega)$ can be written as
$
V_1 = A \cdot V_0
$
for some $A \in U(S)$. In \cite{arnold}, this fact is used to show that $H^1(Lag(S,\omega),\Z) \cong \Z$.
It generator $\mu \in H^1(Lag(S,\omega),\Z)$ is
the {\it Maslov class} \cite{arnold} and two loops $\gamma_1, \, \gamma_2$ are homotopic
if and only if $\mu(\gamma_1) = \mu(\gamma_2)$.
\par
We give an elementary description of the
Maslov class below. We fix $V_0 \in Lag(S,\omega)$ and put
$$
Lag_1(S,\omega) = \{ V \in Lag(S,\omega) \mid
\dim(V \cap V_0) \geq 1 \}.
$$
It is proven in \cite{arnold}
that $Lag_1(S,\omega)$
is {\it co-oriented} and so defines a cycle whose Poincar\'e dual
is precisely the Maslov class $\mu \in H^1(Lag_1(S,\omega),\Z)$.
\par
The tangent space $T_{V_0}Lag(S,\omega)$ is
canonically isomorphic to the set of quadratic forms on $V_0$.
\begin{defn}\label{defn:+directed}
We say any tangent vector pointing
the chamber of nondegenerate positive-definite quadratic forms
is \emph{positively directed.}
\end{defn}
The following is also proved in \cite{arnold}.
\begin{lem}\label{2.1.1}
There exists a neighborhood $U$ of $V_0 \in Lag(S,\omega)$, the set
$$
U \setminus Lag_1(S,\omega;V_0)
$$
has exactly $n+1$ connected components each of which contains $V_0$ in
its closure.
\end{lem}
We refer readers to \cite{arnold} or see Proposition 3.3 of
\cite{fooo06} for the proof of the following proposition.
\begin{prop}\label{2.1.3} Let $(S,\omega)$ be a symplectic vector space
and $V_0\in (S,\omega)$ be a given
Lagrangian subspace. Let $V_1 \in Lag(S,\omega)
\setminus Lag_1(S,\omega;V_0)$
i.e., be a Lagrangian subspace with $V_0 \cap V_1 = \{0\}$.
Consider smooth paths $\alpha: [0,1] \to Lag(S,\omega)$ satisfying
\begin{enumerate}
\item $\alpha(0) = V_0,\, \alpha(1) = V_1$.
\item $\alpha(t) \in Lag(S,\omega)
\setminus Lag_1(S,\omega;V_0)$
for all $ 0 < t \leq 1$.
\item $\alpha'(0)$ is positively directed.
\end{enumerate}
Then any two such paths $\alpha_1, \, \alpha_2$ are homotopic to
each other via a homotopy $s \in [0,1] \mapsto \alpha_s$ such that
each $\alpha_s$ also satisfies the $3$ conditions above.
\end{prop}
Let
$
Lag^+(S,\omega)
$
be the double cover of $Lag(S,\omega)$. Its element is regarded as an
element $V$ of $Lag(S,\omega)$ equipped with an orientation of $V$.
\subsection{Anchors and grading}
\label{subsec:anchorgrading}
To use the anchor in the definition of a grading in the Floer
complex, we need to equip each anchor with an additional decoration.
\par
Let $y \in M$ be the base point. We fix an oriented Lagrangian
subspace $V_y \in Lag^+(T_yM)$.
\begin{defn} \label{ancgrade}
Consider an anchored Lagrangian $(L,\gamma)$. We denote by
$\lambda$ a section of $\gamma^*Lag^+(M,\omega)$
such that
$$
\lambda(0)= V_y, \quad \lambda(1) = T_{\gamma(1)}L.
$$
We call such a pair $(\gamma,\lambda)$
a \emph{graded anchor} of $L$ (relative to $(y,V_y)$) and
a triple $(L,\gamma,\lambda)$ a \emph{graded anchored Lagrangian submanifold}.
\end{defn}
\begin{rem}
We remark that a notion similar to the graded anchor also appears in
Welchinger's recent work \cite{Welschinger}.
\end{rem}

Let $(L_0,\gamma_0,\lambda_0)$ and $(L_1,\gamma_1,\lambda_1)$ be graded
anchored Lagrangian submanifolds relative to
$(y,V_y)$. Assume that
$L_0$ and $L_1$ intersect transversely.
We define $\lambda_{01}(t) \in Lag^+(T_{\ell_{01}(t)}M)$
by concatenating $\lambda_0$ and $\lambda_1$ as follows:
\begin{equation}\label{lamda01}
\lambda_{01}(t)
=
\begin{cases}
\lambda_0(1-2t) & t\le 1/2 \\
\lambda_1(2t-1) & t\ge 1/2.
\end{cases}
\end{equation}
We consider a pair $[p,w]$ where $p \in L_0 \cap L_1$, and $w: [0,1]^2 \to M$
as in (\ref{condw}).
To put a grading at $[p,w]$, we recall the definition of Maslov-Morse index
introduced in \cite{fooo00}. For given $w$, we
associate a Lagrangian loop $\alpha_{[p,w];\lambda_{01}}$ defined on
$\partial [0,1]^2$ by
\begin{equation}
\aligned
\alpha_{[p,w];\lambda_{01}}(0,t) & = \lambda_{01}(t), \qquad
\alpha_{[p,w];\lambda_{01}}(s, 0) \equiv T_{w(s,0)}L_0, \\
\alpha_{[p,w];\lambda_{01}}(s,1) &\equiv T_{w(s,1)}L_1, \quad
\alpha_{[p,w];\lambda_{01}}(1,t) = \alpha_p^+(t)
\endaligned
\label{loopboundary}
\end{equation}
where $\alpha_p^+:[0,1] \to T_pM$ is a path connecting from
$T_pL_0$ to $T_pL_1$ in $Lag(T_pM,\omega_p)$ whose homotopy class is
the unique one as described in Proposition \ref{2.1.3}.
\par
Let $p \in L_0 \cap L_1$ and $w: [0,1]^2 \to M$ satisfy (\ref{condw}).
Choose a symplectic trivialization $
\Phi = (\pi, \phi) : w^*TM \to [0,1]^2 \times T_pM \cong [0,1]^2 \times \R^{2n} $
where $\pi: w^*TM \to [0,1]^2$ and $\phi:w^*TM \to T_pM$ are
the corresponding projections to $[0,1]^2$ and $T_pM$ respectively.
$\Phi$ is homotopically unique. Now we denote by
$\alpha_{[p,w];\lambda_{01}}^\Phi$ the Lagrangian loop
$$
\alpha_{[p,w];\lambda_{01}}^\Phi = \phi(\alpha_{[p,w];\lambda_{01}}\circ c).
$$
Here we fix a piecewise
smooth parametrization $c:S^1\cong \R/\Z \to \del [0,1]^2$ of $\del [0,1]^2$ with positive orientation
with $c(0) = (1,0)$.
\begin{defn}\label{MMindex2}
We define the Maslov-Morse index, denoted by $\mu([p,w];\lambda_{01})$,
to be the Maslov index of this Lagrangian
loop $\alpha_{[p,w];\lambda_{01}}^\Phi $ in $(T_pM, \omega)$.
\end{defn}
This definition does not depend on the trivialization $\Phi$ or on the
(positive) parametrization $c$ of $\del[0,1]^2$ and so well-defined.
\begin{rem}
Here and hereafter we uses the symbol $\lambda$ for a path in $Lag^+$, the oriented
Lagrangian Grassmannian and $\alpha$ for a path in $Lag$, the un-oriented
Lagrangian Grassmannian.
\end{rem}
\begin{lem}\label{degPDlem}
Let $p,w,\lambda_{01}$ be as in Definition $\ref{MMindex2}$.
We put
$$
w^-(s,t) = w(s,1-t), \qquad \lambda_{10}(t) = \lambda_{01}(1-t).
$$
Then
\begin{equation}\label{degPD}
\mu([p,w];\lambda_{01}) + \mu([p,w^-];\lambda_{10}) = n.
\end{equation}
\end{lem}
\begin{proof} Let $\Phi$ and $\alpha^\Phi_{[0,1];\lambda_{01}}$ be
as above. If we denote $\iota: [0,1]^2 \to [0,1]^2$ to be the map
$\iota(s,t) = (s,1-t)$, we have $w^- = w\circ \iota$.
Therefore we can trivialize $(w^-)^*TM = \iota^*w^*TM$
by the map $\Phi^-: (w^-)^*TM \to [0,1]^2 \times \R^{2n}$
defined by $\Phi^- = \Phi \circ \iota^*$.
Then
$$
\widetilde{\alpha_{[p,w];\lambda_{01}}^\Phi}=
\phi(\alpha_{[p,w];\lambda_{01}}\circ \widetilde c)
$$
where $\widetilde{(\cdot)}$ denotes the inverse path, e.g., $\widetilde c(\theta) = c(-\theta)$.
By definition of $\alpha_{[p,w^-];\lambda_{10}}^{\Phi^-}$, the path
$\phi(\alpha_{[p,w];\lambda_{01}}\circ \widetilde c)$
coincides with $\alpha_{[p,w^-];\lambda_{10}}^{\Phi^-}$
(up to parametrization)
except on the segment $c^{-1}(\{1\} \times [0,1])$.
\par
Therefore the composition $\alpha_{[p,w];\lambda_{01}}^\Phi* \alpha_{[p,w^-];\lambda_{10}}^{\Phi^-}$
of $\alpha_{[p,w^-];\lambda_{10}}^{\Phi^-}$ and
$\alpha_{[p,w];\lambda_{01}}^\Phi$ is homotopic to a path $\alpha = \alpha^- \cup \alpha^+ : S^1 \cong I^- \cup I^+
\to Lag(\R^{2n},\omega_0)$ with $I^\pm =\{1\} \times [0,1]$ such that both $\alpha^\pm:I^\pm \to
Lag(\R^{2n},\omega_0)$ are the paths
positively directed at $t = 0$ provided in Proposition \ref{2.1.3} and satisfy
$$
\alpha^+(0) = \alpha^-(1) = \phi(T_pL_0), \quad \alpha^+(1) = \alpha^-(0) = \phi(T_pL_1).
$$
It is easy to see that the Maslov index of such
$\alpha$ is $n$ and hence we obtain
\be\label{eq:muw-}
n = \mu(\alpha_{[p,w];\lambda_{01}}^\Phi* \alpha_{[p,w^-];\lambda_{10}}^{\Phi^-})
= \mu(\alpha_{[p,w];\lambda_{01}}^\Phi) +
\mu(\alpha_{[p,w^-];\lambda_{10}}^{\Phi^-}).
\ee
By definition, the last sum is nothing but $\mu([p,w];\lambda_{01}) + \mu([p,w^-];\lambda_{10})$.
This finishes the proof of \eqref{degPD}.
\end{proof}
\subsection{Polygonal Maslov index}
\label{subsec:polygonal}
Consider a chain of Lagrangian submanifolds $\frak L = (L_0, \cdots, L_k)$ and
a chain of intersection points $(p_{0k},p_{k(k-1)},\cdots, p_{10})$ with $p_{i(i-1)} \in
L_{i-1}\cap L_i$ for $i=0, \cdots, k$. We consider the disc with
marked points $\{z_{0k},z_{k(k-1)}, \cdots, z_{10}\}$ and denote $\dot D^2 = D^2
\setminus \{z_{0k},z_{k(k-1)},\cdots, z_{10}\}$.
We assume
$z_{0k},z_{k(k-1)}, \cdots, z_{10}$ respects counter clock-wise cyclic order of
$\partial D^2$.
\begin{rem}
Here and hereafter the suffix $j$ is regarded as modulo $k+1$. Namely
$p_{(j+1)j}$ in case $j=k$ means $p_{0k}$, for example. We also put
$p_{ij} = p_{ji}$.
\end{rem}

For the following discussion, we will consider the cases $k \geq 1$,
i.e, the cases with $\operatorname{length} \frak L \geq 2$.
\par
For each given such chains, we define the set of maps
$$
C^\infty(D^2,\CE;\vec p), \quad \vec p =\{p_{0k},p_{k(k-1)},\cdots, p_{10}\}
$$
to be the set of all $w: D^2 \to M$ such that
\begin{equation}\label{eq:wzjxj}
w(\overline{z_{(j+1)j}z_{j(j-1)}}) \subset L_j, \quad w(z_{j(j-1)})= p_{j(j-1)} \in L_j
\cap L_{j-1},
\end{equation}
and that it is continuous on $D^2$ and smooth on $\dot D^2$. We will
define a topological index, which is associated to each homotopy
class $B \in \pi_2(\frak L;\vec p)$. We denote it by $\mu(\frak L,
\vec p;B)$. Let $w\in C^\infty(\dot D^2,\frak L;\vec p)$ be a map
such that $[w] = B$. We denote by
$$
\CF(\frak L;\vec p; B) \subset
C^\infty(D^2,\frak L;\vec p)
$$
the set of such maps.
\par
We identify $[0,2\pi]/(0\sim 2\pi) \cong S^1$ by $t \mapsto e^{\sqrt{-1}t}$.
(The direction $t$ increase then becomes counter-clockwise order of $S^1$.)
\par
Under a symplectic trivialization of the bundle $w^*TM$,
 the map
$$
\alpha_w: S^1=\del D^2 \to Lag(\R^{2n},\omega_0); \qquad
t \mapsto T_{w(t)}L_i
\quad\mbox{if $t \in \overline{z_{(i+1)i}z_{i(i-1)}}$}
$$
defines a piecewise smooth path with discontinuities at $(k+1)$
points $z_{i(i-1)}\in \del D^2$ for $i = 0, 1, \cdots, k$, at which
we have
\begin{equation}\label{58.5}
\lim_{t \to z_{i(i-1)}-0}\alpha_w(t) = T_{p_{i(i-1)}}L_i,
\quad \lim_{t \to z_{i(i+1)}+ 0}\alpha_w(t) =
T_{p_{i(i-1)}}L_{i-1}
\end{equation}
\par
By the transversality hypothesis,
$T_{p_{i(i-1)}}L_i$ and $
T_{p_{i(i-1)}}L_{i-1}$ are Lagrangian subspaces in
$(\R^{2n},\omega_0)$ with $T_{p_{i(i-1)}}L_i \cap T_{p_{i(i-1)}}L_{i-1} = \{0\}$. We fix
a smooth path $\alpha^-_{i(i-1)}:[0,1] \to Lag(\R^{2n},\omega_0)$
for each $i = 0, \cdots, k$ so that
\begin{equation}\label{58.6}
\alpha^-_{i(i-1)}(0) = T_{p_{i(i-1)}}L_i,
\quad \alpha_{i(i-1)}^-(1)= T_{p_{i(i-1)}}L_{i-1}
\end{equation}
and $-(\alpha^-_{i(i-1)})'(1)$ is positively directed in the sense of
Definition \ref{defn:+directed}.
\par
In other words $\alpha_{i(i-1)}^-(t) = \alpha_{p_{(i-1)i}}^+(1-t)$, where
the right hand side is as in (\ref{loopboundary}).
By Proposition \ref{2.1.3},
such a choice is unique up to homotopy relative to the end points
$t=0, \,1$.
Inserting $\alpha^-_{i(i-1)}$ into the map $\alpha_w$ at
each $z_{i(i-1)}$, we obtain a continuous loop $\widetilde
\alpha_w$ in $Lag(\R^{2n},\omega_0)$.
\begin{defn}\label{58.7}
Let $\frak L=(L_0, \cdots, L_k)$ be a Lagrangian chain.
We define the topological index, denoted by $\mu(\frak L, \vec p;B)$,
to be the Maslov index of
the loop $\widetilde\alpha_w$, i.e.,
$$
\mu(\frak L, \vec p;B) = \mu(\widetilde\alpha_w).
$$
\end{defn}
This definition is essentially reduced to the Maslov-Viterbo index \cite{viterbo}
for the pairs $(L_0,L_1)$ when $k = 1$ and
reduces to the one given in Section A3 \cite{KaSh90} for the
case where $L_i$ are all affine.
\begin{rem}
We remark that $L_0, L_1, \ldots, L_k$ are put on the boundary
of the disc $D^2$ in a clockwise order.
On the other hand, $z_{0k},z_{k(k-1)},\ldots,z_{21}, z_{10}$ are in the
counter clockwise order.
\par
This is consistent with the case $k=1$ discussed in \cite{fooo08}.
See Remark 3.7.23 (1) \cite{fooo08}.
\end{rem}
Now consider a chain of graded anchored Lagrangian submanifolds
$\CE = (\CL_0, \cdots, \CL_k)$, $\CL_i = (L_i,\gamma_i,\lambda_i).$
It induces a grading $ \lambda_{ij} $ along $\ell_{ij}$ as in
(\ref{lamda01}). We remark that $\lambda_{ij}$ also satisfy the
overlapping property
\begin{equation}
\lambda_{ij}|_{[0,\frac{1}{2}]} = \lambda_{i\ell}|_{[0,\frac{1}{2}]}\quad
\lambda_{ij}|_{[\frac{1}{2},1]} = \lambda_{\ell
j}|_{[\frac{1}{2},1]}.
\end{equation}
Let $p_{(i+1)i}=p_{i(i+1)} \in L_i \cap L_{i+1}$. We put
\begin{equation}\label{wplus}
w^+_{(i+1)i}(s,t) = w^-_{i(i+1)}(1-s,t)
\end{equation}
where the right hand side is as in Definition \ref{classB}.
\begin{lem}\label{thm:poly} Let $\CE$ be a graded anchored Lagrangian chain.
Suppose $B \in \pi_2^{ad}(\CE,\vec p)$
given as Lemma {\rm\ref{classB}}.
Then we have
\begin{equation}\label{musum}
\mu(\frak L,\vec p;B) + \sum_{i=0}^{k}
\mu([p_{(i+1)i},w^+_{(i+1)i}];\lambda_{i(i+1)}) = 0.
\end{equation}
\end{lem}
\begin{proof}
Since $\mu([p_{(i+1)i},w_{(i+1)i}^+];\lambda_{i(i+1)})$ is defined as the Maslov index of
the loop $\alpha_{[p_{(i+1)i},w_{(i+1)i}^+];\lambda_{i(i+1)}}$  (Definition \ref{MMindex2}), the equality (\ref{musum}) follows
from
$$
\sum_{i=0}^k \alpha_{[p_{(i+1)i},w_{(i+1)i}^+];\lambda_{i(i+1)}} + \widetilde{\alpha}_w \sim 0,
$$
where $\widetilde\alpha_w$
is as in Definition \ref{58.7}
with $B = [w^-_{01}]\# \cdots \# [w^-_{(k-1)k}] \# [w^-_{k0}]$ and
$\sim$ means homologous.
\end{proof}

When the length of $\CE$ is $k+1$, we define
$
\mu_k(B) = \mu(\CE,\vec v;B)
$
where $B \in \pi_2^{ad}(\CE;\vec p)$.
\begin{cor}
Define $\mu_1:\pi_2(\ell_{01},p) \to \Z$ by setting
$
\mu_1(\alpha) : = -\mu([p,w];\lambda_{01})
$
for a representative $[p,w]$ of the class $\alpha \in \pi_2(\ell_{01};p)$.
Then the sequence of maps $\mu = \{\mu_k\}_{k=1}^\infty$ with
$
\mu_k: \pi_2^{ad}(\CE;\vec p) \to \Z$, $k \geq 1
$
defines an abstract index.
\end{cor}
\section{Orientation}
\label{sec:orient}
To be able to define various operators in Floer theory, we need to provide a
compatible system of orientations on the Floer moduli spaces and other moduli spaces
of pseudo-holomorphic polygons. We here explain the way to give orientations, which is basically the same as
\cite{fooo06}.
\begin{defn} A submanifold $L \subset M$ is called
\emph{relatively spin} if it is orientable and
there exists a class $st \in H^2(M,\Z_2)$ such that
$st|_L = w_2(TL)$ for the Stiefel-Whitney class $w_2(TL)$ of $TL$.
\par
A chain $(L_0,L_1,\cdots,L_k)$ or a pair $(L_0,L_1)$ of Lagrangian submanifolds is
said to be relatively spin if there exists a class
$st \in H^2(M,\Z_2)$ satisfying $st|_{L_i} = w_2(TL_i)$ for each
$i = 0, 1, \cdots, k$.
\end{defn}
We fix such a class $st \in H^2(M,\Z_2)$ and a triangulation of $M$.
Denote by $M^{(k)}$ its $k$-skeleton. There exists a
real vector bundle $V(st)$ on $M^{(3)}$ with
$w_1(V(st)) = 0, \, w_2(V(st)) = st$. Now suppose that $L$ is
relatively spin and $L^{(2)}$ be the 2-skeleton of $L$.
Then $V\oplus TL$ is trivial on the 2-skeleton
of $L$. We define
\begin{defn} We define a $(M,st)$-relative
spin structure of $L$ to be a choice of $V$ and a spin structure of the
restriction of the vector bundle $V\oplus TL$ to $L^{(2)}$.
\par
The relative spin structure of a chain of Lagrangian submanifolds
$(L_0,\cdots,L_k)$ is defined in the same way by
using the same $V$ for all $L_i$.
\end{defn}
Let $p, q \in L_0\cap L_1$ and $B \in \pi_2(p,q)$. We consider $ u
: \R\times [0,1] \to M $ such that
\begin{subequations}\label{2gonmodulidefine}
\begin{eqnarray}
&{}&\frac{du}{d\tau} + J \frac{du}{dt} = 0 \\
&{}&u(\tau,0) \in L_0, \quad u(\tau,1) \in L_1, \, \int u^*\omega < \infty \\
&{}&u(-\infty,\cdot) \equiv p, \quad u(\infty,\cdot) \equiv q.
\end{eqnarray}
\end{subequations}
It induces a continuous map $\overline u: [0,1]^2 \to M$ with
$\overline u(0,t) \equiv p, \, u(1,t) \equiv q$ in an obvious way.
With an abuse of notation, we denote by $[u]$ the homotopy class of
the map $\overline u$ in $\pi_2(p,q)$. We denote by
$\widetilde\CM^{\circ}(p,q;B)$ the moduli space consisting of the
maps $u$ satisfying (\ref{2gonmodulidefine}) and compactify
$\widetilde\CM^{\circ}(p,q;B)/\R$ its quotient by the
$\tau$-translations by using an appropriate notion of stable maps as
in  Section 3 \cite{fooo00}. We denote the compactification by
$\CM(p,q;B)$. We call this the Floer moduli space. It carries the
structure of a space with Kuranishi structure.
\par
If $(L_0, L_1)$ is a relatively spin pair, then $\CM(p,q;B)$ is
orientable. Furthermore a choice of relative spin structures gives
rise to a compatible system of orientations for $\CM(p,q;B)$ for all
pair $p, \, q \in L_0 \cap L_1$ and $B \in \pi_2(p,q)$. For
completeness' sake, we now recall from \cite{fooo06} how the
relative spin structure gives rise to a system of coherent
orientations.
\par
Let $p \in L_0 \cap L_1$ and $w$ satisfies (\ref{condw}) . We denote
by $Map(\ell_{01};p;L_0,L_1;\alpha)$ the set of such maps $[0,1]^2
\to M$ its homotopy class $[w]=\alpha$ in $\pi_2(\ell_{01};p)$. Let
$w \in Map(\ell_{01};p;L_0,L_1;\alpha)$. Let $
\Phi : w^*TM \to [0,1]^2 \times T_pM $ be a (homotopically unique) symplectic
trivialization as before. The trivialization $\Phi$,
together with the boundary condition, $w(0,t) = \ell_{01}(t)$ and
the Lagrangian path $\lambda_{01}$ along $\ell_{01}$, defines a
Lagrangian path
$$
\lambda^\Phi=\lambda^\Phi_{([p,w];\lambda_{01})}:[0,1] \to T_pM
$$
satisfying $\lambda^\Phi(0) = T_pL_0, \, \lambda^\Phi(1) = T_pL_1$.
The homotopy class of this path does not depend on the
trivialization $\Phi$ but depends only on $[p,w]$ and the homotopy
class of $\lambda_{01}$. Hereafter we omit $\Phi$ from notation.
\par
We remark that relative spin structure determines a trivialization
of $V_{\lambda_{01}(0)} \oplus T_{\lambda_{01}(0)}L_0 =
V_{\lambda_{01}(0)} \oplus \lambda_{01}(0)$ and $V_{\lambda_{01}(1)}
\oplus T_pL_1 = V_{\lambda_{01}(1)} \oplus \lambda_{01}(1)$. We take
and fix away to extend this trivialization to the family
$\ell_{01}^*V \oplus \lambda_{01}$ on $[0,1]$.
\par
We consider the following boundary valued problem for the
section $\xi$ of $w^*TM$ on $\R_{\ge 0} \times [0,1]$ of $W^{1,p}$ class such that:
\begin{subequations}\label{eq:RH}
\begin{eqnarray}
&{}&D_w\delbar (\xi) = 0 \label{eq:RHequa}\\
&{}& \xi(0,t) \in \lambda_{01}(t), \quad \xi(\tau,0) \in T_pL_0, \,
\xi(\tau,1) \in T_pL_1. \label{eq:RHbod}
\end{eqnarray}
\end{subequations}
Here $D_w\delbar$ is the linearization operator of the Cauchy-Riemann equation.
\par
We define $W^{1,p}(\R_{\ge 0} \times [0,1],T_pM ; \lambda_{01})$ to be the set
of sections $\xi$ of $w^*TM$ on $\R_{\ge 0} \times [0,1]$ of $W^{1,p}$ class
satisfying (\ref{eq:RHbod}).
Then (\ref{eq:RHequa}) induces an operator
$$
D_w\delbar : W^{1,p}(\R_{\ge 0} \times [0,1],T_pM ; \lambda) \to
L^p(\R_{\ge 0} \times [0,1],T_pM\otimes \Lambda^{0,1}),
$$
which we denote by $\delbar_{([p,w];\lambda_{01})}$.
The following
proposition was proved in Lemma 3.7.69 \cite{fooo08}.
\begin{prop}\label{prop:delbarindex} We have
\be\label{eq:index} \operatorname{Index} \delbar_{([p,w];\lambda_{01})} =
\mu([p,w];\lambda_{01}). \ee
\end{prop}
We denote its determinant line by
$$
\det \delbar_{([p,w];\lambda_{01})}.
$$
By varying $w$ in its homotopy class $\alpha \in \pi_2(\ell_{01};p)
= \pi_2(\ell_{01};p;L_0,L_1)$, these lines define a line bundle
\begin{equation}\label{eq:detB1}
\det\delbar_{([p,w];\lambda_{01})} \to
Map(\ell_{01};p;L_0,L_1;\alpha).
\end{equation}
The bundle (\ref{eq:detB1}) is trivial if $(L_0,L_1)$ is a
relatively spin pair. (See  Section 8.1 \cite{fooo08}.)
\par
We need to find a systematic way to orient (\ref{eq:detB1}) for
various $\alpha \in \pi_2(\ell_{01};p)$ simultaneously. Following
Subsection 8.1.3 \cite{fooo08} we proceed as follows.
Let $\lambda_p:[0,1] \to T_pM$ be a path
connecting from $T_pL_0$ to $T_pL_1$ in $Lag^+(T_pM,\omega)$.
The relative spin structure determines a trivialization of
$V_p \oplus T_pL_0 = V_p \oplus \lambda_p(0)$ and of $V_p \oplus T_pL_1 = V_p \oplus \lambda_p(1)$.
We fix an extension of this trivialization of the $[0,1]$ parametrized
family of vector spaces $V_p \oplus \lambda_p$.
We define
\begin{equation}
Z_+ = \{ (\tau,t) \in \R^2 \mid \tau \le 0, \,\, 0\le t\le 1\} \cup
\{(\tau,t) \mid \tau^2 + (t-1/2)^2 \leq 1/4\}
\end{equation}
\begin{rem} We would like to remark that attaching the semi-disc to the side
of the semi-strip $t = 0$ is not necessary for the definition of
$Z_+$. However for the consistency with the notation of Subsection
8.1.3 \cite{fooo08}, we keep using $Z_+$ instead of the simpler
$(-\infty,0] \times [0,1]$.
\end{rem}
We consider maps $\xi : Z_+ \to T_pM$ of $W^{1,p}$ class and study
the linear differential equation
\begin{subequations}\label{eq:RHZ}
\begin{eqnarray}
&{}&\delbar\xi = 0 \label{eq:RHequaZ}\\
&{}&\xi( e^{\pi i (t-1/2)}/2 +i/2) \in \lambda_p(t), \,\, \xi(\tau,0) \in T_pL_0, \,\,
\xi(\tau,1) \in T_pL_1. \label{eq:RHbodZ}
\end{eqnarray}
\end{subequations}
It defines an operator
$$
W^{1,p}(Z_+,T_pM;\lambda_p) \to L^p(Z_+;T_pM\otimes \Lambda^{0,1} ),
$$
which we denote by $\delbar_{\lambda_p}$.
Let $\operatorname{Index}\,\delbar_{\lambda_p}$ be its index, which is a
virtual vector space.
The following theorem is proved in the same was Chapter 8 \cite{fooo06}.
\begin{thm}\label{thm:fooo-ori}
Let $(L_0,L_1)$ be a relatively spin pair of oriented Lagrangian
submanifolds. Then for each fixed $\alpha$ the bundle
\eqref{eq:detB1} is trivial.
\par
If we fix a choice of system of
orientations $o_{p}$ on $\operatorname{Index}\,\delbar_{\lambda_p}$
for each $p$, then it
determines orientations on $(\ref{eq:detB1})$, which we denote by $o_{[p,w]}$.
\par
Moreover $o_{p}$, $o_{[p,w]}$ determine an orientation of
$\CM(p,q;B)$ denoted by $o(p,q;B)$ by the gluing rule
\be\label{eq:ori-gluing}
o_{[q,w\#B]} = o_{[p,w]} \# o(p,q;B)
\ee
for all $p, \, q \in L_0 \cap L_1$ and $B \in \pi_2(p,q)$ so that
they satisfy the gluing formulae
$$
\del o(p,r;B) = o(p,q;B_1) \# o(q,r;B_2)
$$
whenever the virtual dimension of $\CM(p,r;B)$ is $1$.  Here
$\del o(p,r;B)$ is the induced boundary orientation of
the boundary $\del \CM(p,r;B)$ and
$B = B_1 \# B_2$ and $\CM(p,q;B_1) \# \CM(q,r;B_2)$ appears
as a component of the boundary $\del \CM(p,r;B)$.
\end{thm}

\begin{rem}  \label{connorbori}
In the last statement in Theorem \ref{thm:fooo-ori}, we assumed that
$\CM(p,r;B)$ is one-dimensional.  In general, we have
$$
\del o(p,r;B) = (-1)^{\epsilon} o(p,q;B_1) \# o(q,r;B_2),
$$
where $\epsilon= \dim \CM(q,r;B_2)$, which is presented in the proof of
Proposition 8.7.3 in \cite{fooo08}.
For the definition of  the orientation of the moduli spaces for the filtered bimodule structure,
see sections 8.7 and 8.8 (Definition 8.8.11) in \cite{fooo08}.
\end{rem}

\begin{proof}
The first paragraph follows from Section 8.1 \cite{fooo08}. We glue
the end $(-\infty,t)$ of $Z_+$ with $(+\infty,t)$ of $\R_{\ge 0}
\times [0,1]$ to obtain $(\R_{\ge 0} \times [0,1]) \# Z_+$. We
`glue' operators (\ref{eq:RH}) and (\ref{eq:RHZ}) in an obvious way
to obtain an operator $D_w\delbar_{(\R_{\ge 0} \times [0,1]) \#
Z_+}$ on $(\R_{\ge 0} \times [0,1]) \# Z_+$. We have an isomorphism
of (family of) virtual vector spaces:
\begin{equation}
\text{\rm Index}\,(D_w\delbar_{(\R_{\ge 0} \times [0,1]) \# Z_+})
\cong
\text{\rm Index}\,\delbar_{([p,w];\lambda_{01})} \oplus
\text{\rm Index}\,\delbar_{\lambda_p}
\end{equation}
We fixed a trivialization of the family of vector spaces $V_p \oplus
\lambda_p$ and $\ell_{01}^*V \oplus \lambda_{01}$, which extends a
trivialization of $V\oplus TL_0$, $V \oplus TL_1$ on the two
skeletons of $L_0$ and $L_1$ respectively, which is given by the
relative spin structure. It induces a canonical orientation of the
index bundle $\text{\rm Index}\,(D_w\delbar_{(\R_{\ge 0} \times
[0,1]) \# Z_+})$ by Lemma 3.7.69 \cite{fooo08}. Therefore the
orientation of $\text{\rm Index}\,\delbar_{\lambda_p}$ induces an
orientation of $\text{\rm Index}\,\delbar_{([p,w];\lambda_{01})}$ in
a canonical way. This implies the second paragraph.
\par
The third paragraph is a consequence of \eqref{eq:ori-gluing}
which is similar to the proof of Theorem 8.1.14 \cite{fooo08}.
\end{proof}
One can generalize the above discussion to the moduli space of
pseudo-holomorphic polygons in a straightforward way, which
we describe below.
\par
Consider a disc $D^2$ with $k+1$ marked points $z_{0k}, z_{k(k-1)}, \cdots, z_{10}
\subset \del D^2$ respecting the counter clockwise cyclic order of
$\del D^2$. We take a neighborhood $U_i$ of $z_{i(i-1)}$ and a conformal
diffeomorphism $ \varphi_i: U_i \setminus \{z_{i(i-1)}\} \subset D^2 \cong
(-\infty,0] \times [0,1] $ of each $z_{i(i-1)}$. For any smooth map
$$
w: D^2 \to M; w(z_{i(i-1)}) = p_{i(i-1)}, \, w(\overline{z_{(i+1)i}z_{i(i-1)}}) \subset L_i
$$
we deform $w$ so that it becomes constant on
$\varphi_i^{-1}((-\infty,-1] \times [0,1]) \subset U_i$, i.e., $ w(z)
\equiv p_{i(i-1)} $ for all $z \in \varphi^{-1}((-\infty,-1] \times
[0,1])$. So assume this holds for $w$ from now on. We now consider
the Cauchy-Riemann equation
\begin{subequations}\label{eq:polyCR}
\begin{eqnarray}
&{}& D_w\delbar (\xi) = 0 \\
&{}& \xi(\theta) \in T_{w(t)}L_i \quad \mbox{for } \, \theta
\in \overline{z_{(i+1)i}z_{i(i-1)}} \subset \del D^2. \label{6.8.b}
\end{eqnarray}
\end{subequations}
We remark that on $U_i = (-\infty,0] \times [0,1]$ the boundary condition (\ref{6.8.b}) becomes
\begin{equation}\label{bdryui}
\xi(s,0) \in L_{i-1}, \quad \xi(s,1) \in L_{i}.
\end{equation}
\par
(\ref{eq:polyCR}) induces a Fredholm operator, which we denote
by
\begin{equation}\label{delbawl}
\delbar_{w;\frak L} :
W^{1,p}(D^2;w^*TM;\frak L) \to L^p(D^2;w^*\otimes \Lambda^{0,1}).
\end{equation}
Moving $w$ we obtain a family of Fredholm operator $\delbar_{(\frak
L;\vec p;B)}$ parametrized by a suitable completion of $\CF(\vec
p;\frak L;B)$ for $B \in \pi_2(\vec p ; \frak L)$. Therefore we have
a well-defined determinant line bundle \be\label{eq:detB} \det
\delbar_{(\frak L;\vec p;B)} \to \CF(\frak L;\vec p;B). \ee The
following theorem is an extension of the above Theorem
\ref{thm:fooo-ori}.
\begin{thm} \label{thm:fooo-ori2}
Suppose $\frak L = (L_0,\cdots,
L_k)$ is a relatively spin Lagrangian chain. Then each $\det
\delbar_{(\vec p;\frak L;B)}$ is trivial.
\par
Moreover we have the following: If we fix orientations $o_{p_{ij}}$
on $\operatorname{Index}\,\delbar_{\lambda_{p_{ij}}}$ as in Theorem
$\ref{thm:fooo-ori}$ for all $p_{ij} \in L_i \cap L_j$, with $L_i$
transversal to $L_j$, then we have a
system of orientations, denoted by $o_{k+1}(\vec p;\frak L;B)$, on the
bundles $(\ref{eq:detB})$ so that it is compatible with gluing map
in an obvious sense.
\end{thm}
\begin{proof}
Let $w^+_{(i+1)i} \in \pi_2(\ell_{i(i+1)};p_{(i+1)i})$ be as in
(\ref{wplus}) and we consider the operator
$\delbar_{([p_{(i+1)i},w^+_{(i+1)i}];\lambda_{i(i+1)})}$. We glue it
with $ \delbar_{(\frak L;\vec p;B)}$ at $U_{i+1}$. ((\ref{bdryui})
implies that the boundary condition can be glued.) After gluing all
of $\delbar_{([p_{(i+1)i},w^+_{(i+1)i}];\lambda_{i(i+1)})}$ we have
an index bundle of a Fredholm operator
\begin{equation}\label{summedupindex}
\delbar_{(\frak L;\vec p;B)} \# \sum_{i=k}^0 \delbar_{([p_{(i+1)i},w^+_{(i+1)i}];\lambda_{i(i+1)})}.
\end{equation}
By Lemma 3.7.69 \cite{fooo08}, the index bundle of (\ref{summedupindex})
has canonical orientation.
On the other hand, index virtual vector spaces
of $\delbar_{([p_{(i+1)i},w^+_{(i+1)i}];\lambda_{i(i+1)})}$ are oriented by Theorem \ref{thm:fooo-ori}.
Theorem \ref{thm:fooo-ori2} follows.
\end{proof}

We can prove that the orientation of $\delbar_{(\frak L;\vec p;B)}$
depends on the choice of $o_{p_{(i+1)i}}$ (and so on $\lambda_p$)
with $i=0,\cdots,k$ but is independent
of the choice of $w^+_{(i+1)i}$ etc. This is a consequence of the
proof of Theorem \ref{thm:fooo-ori}. (We omit the detail of this point.
See Remark 8.1.15 (3) \cite{fooo08}.)
Therefore the orientation in
Theorem \ref{thm:fooo-ori2} is independent of the choice of anchors.
\par

\begin{rem}\label{fooo8-5}
In order to give an orientation of $\CM(\frak L;\vec p;B)$, we have to take the moduli
parameters of marked points and the action of the automorphism group into account.
We also treat the intersection point $p_{i(i-1)}$ as if it is a chain
of codimension $\mu([p_{i(i-1)},w_{i(i-1)}^+];\lambda_{(i-1)i})$ in a similar way to
Chapter 8, section 8.5 in \cite{fooo08}.
\end{rem}

\section{Floer chain complex}
\label{chaincomplex}
In this subsection, we will describe construction of the boundary
map. We also mention some (minor) modification needed in its construction
in the context with anchored Lagrangian submanifolds.
\par
Let $(L_i,\gamma_i)$ $i=0,1$ be anchored Lagrangian submanifolds.
We write $\CE = ((L_0,\gamma_0),(L_1,\gamma_1))$.
Let $p, \, q \in L_0 \cap L_1$ be admissible
intersection points. We
defined the set $\pi_2(p,q)=\pi_2((L_0,L_1),(p,q))$
in Section \ref{sec:gluinghomotopy}.
We also defined $\pi_2(\ell_{01};p)$ there.
We now define:
\begin{defn}
$CF((L_1,\gamma_1),(L_0,\gamma_0))$
is a free $R$ module over the basis $[p,w]$
where $p \in L_0\cap L_1$ is an admissible intersection points and
$[w] \in \pi_2(\ell_{01};p)$.
\end{defn}
Here $R$ is a ground ring such as $\Q$, $\Z$, $\Z_2$, $\C$ or $\R$.
(The choice $\Z$ or $\Z_2$ requires some additional conditions.)
\begin{rem}\label{labelPiaction}
We remark that the set of $[p,w]$ where $p$ is the admissible intersection
point is identified with the set of
the critical point of the action functional $\mathcal A$ defined on the
Novikov covering space of $\Omega(L_0,L_1;\ell_{01})$.
The group $\Pi(L_0,L_1;\ell_{01})$ defined in Section \ref{sec:grading} acts freely on it
so that the quotient space is the
set of admissible intersection points.
\end{rem}
We next take a grading $\lambda_i$ to $(L_i, \gamma_i)$ as in
Subsection \ref{subsec:anchorgrading}.
It induces a grading of $[p,w]$ given by $\mu([p,w];\lambda_{01})$, which gives the
graded structure on $CF(L_1,L_0;\ell_{01})$
$$
CF(L_1,L_0;\ell_{01}) = \bigoplus_k CF^k(L_1,L_0;\lambda_{01})
$$
where
$
CF^k(L_1,L_0;\lambda_{01}) = \operatorname{span}_R\{[p,w]\mid
\mu([p,w];\lambda_{01}) = k\}.
$
\par
For given $B \in \pi_2(p,q)$, we denote by ${Map}(p,q;B)$
the set of such $w$'s in class $B$.
\par
We summarize the extra structures added in the discussion of
Floer homology for the anchored Lagrangian submanifolds in the following

\begin{sit}\label{pairdeta}
We assume that $(L_0,L_1)$ is a relatively spin pair.
We consider a pair $(L_0,\gamma_0)$, $(L_1,\gamma_1)$ of
anchored Lagrangian submanifolds and the base path
$\ell_{01} = \overline\gamma_0*\gamma_1$. We
fix a grading $\lambda_i$ of $\gamma_i$ for $i=0, \, 1$,
which in turn induce a grading of $\ell_{01}$,
$\lambda_{01} = \overline{\lambda_0} * \lambda_1$.
We also fix an orientation $o_{p}$ of $\operatorname{Index}\,\delbar_{\lambda_{p}}$
for each $p\in L_0\cap L_1$ as in Theorem \ref{thm:fooo-ori}.
\end{sit}
We sometime do not explicitly write these extra data in our notations below
as long as there is no danger of confusion.
\par
Let us consider Situation \ref{pairdeta}.
Orientations of the Floer moduli space $\CM(p,q;B)$
is induced by Theorem \ref{thm:fooo-ori}. Using virtual fundamental chain
technique we can take a system of multisections and obtain a system of \emph{rational} numbers
$n(p,q;B) = \#(\CM(p,q;B))$
whenever the virtual dimension of $\CM(p,q;B)$ is zero.
Finally we define the Floer `boundary' map $\partial : CF(L_1,L_0;\ell_{01}) \to
CF(L_1,L_0;\ell_{01})$ by the sum
\begin{equation}\label{eq:boundary}
\partial ([p,w]) = \sum_{q \in L_0\cap L_1}\sum_{B \in \pi_2(p,q)}
n(p,q;B) [q,w\# B].
\end{equation}
By Remark \ref{labelPiaction}, $CF(L_1,L_0;\ell_{01})$ carries a natural
$\Lambda(L_0,L_1;\ell_{01})$-module structure and $CF^k(L_1,L_0;\lambda_{01})$
a $\Lambda^{(0)}(L_0,L_1;\ell_{01})$-module structure where
$$
\Lambda^{(0)}(L_0,L_1;\ell_{01}) = \left\{ \sum a_g [g] \in
\Lambda(L_0,L_1;\ell_{01}) \Big| \mu([g]) = 0 \right\}.
$$
We define
\begin{equation}\label{CFpair}
C(L_1,L_0;\ell_{01})
= CF(L_1,L_0;\ell_{01}) \otimes_{\Lambda(L_0,L_1;\ell_{01})}
\Lambda_{nov}
\end{equation}
where we use the embedding (\ref{maptouniN}).
\par
We write the $\Lambda_{nov}$ module (\ref{CFpair}) also as
$$
C((L_1,\gamma_1),(L_0,\gamma_0);\Lambda_{nov}).
$$
\begin{defn}\label{efilt}
We define the {\it energy filtration}
$
F^{\lambda}CF((L_1,\gamma_1),(L_0,\gamma_0))
$
of the Floer chain complex
$CF(L_1,\gamma_1),(L_0,\gamma_0))$ (here $\lambda \in \R$)
such that $[p,w]$ is in $F^{\lambda}CF((L_1,\gamma_1),(L_0,\gamma_0))$
if and only if $\mathcal A([p,w]) \ge \lambda$.
\end{defn}
This filtration also induces a filtration on (\ref{CFpair}).
\begin{rem}
We remark that this filtration depends (not only of the homotopy class of) but also of $\gamma_i$ itself.
\end{rem}
It is easy to see the following from the definition of $\partial$ above:
\begin{lem}\label{filtpres}
$$
\partial \left(F^{\lambda}CF((L_1,\gamma_1),(L_0,\gamma_0))
\subseteq F^{\lambda}CF((L_1,\gamma_1),(L_0,\gamma_0))\right).
$$
\end{lem}
\section{Obstruction and $A_\infty$ structure}
\label{subsec:obstruction}
Let $(L_0,L_1)$ be a relatively spin pair with $L_0$
intersecting $L_1$ transversely and fix a $(M,st)$-relatively spin
structure for the pair $(L_0,L_1)$.
\par
According to the definition
(\ref{eq:boundary}) of the map $\partial$, we have the formula for
its matrix coefficients
\begin{equation}\label{eq:B=B1B2}
\langle \partial\partial [p,w], [r,w\# B] \rangle = \sum_{q \in L_0\cap L_1}
\sum_{B = B_1\# B_2 \in \pi_2(p,r)} n(p,q;B_1)n(q,r;B_2)T^{\omega(B)}
\end{equation}
where $B_1 \in \pi_2(p,q)$ and $B_2 \in \pi_2(q,r)$.

To prove, $\partial \partial = 0$, one needs to prove
$\langle \partial\partial [p,w], [r,w\# B]\rangle = 0$
for all pairs $[p,w], \, [r,w \# B]$.
On the other hand it follows from definition that each summand
$$
n(p,q;B_1)n(q,r;B_2)T^{\omega(B)} = n(p,q;B_1)T^{\omega(B_1)}n(q,r;B_2)T^{\omega(B_2)}
$$
and the coefficient $n(p,q;B_1)n(q,r;B_2)$
is nothing but the number of broken trajectories lying in
$\CM(p,q;B_1) \# \CM(q,r;B_2)$.
This number is nonzero in the general
situation we work with.

To handle the problem of obstruction to $\del\circ \del = 0$ and of
bubbling-off discs in general, a structure of filtered $A_\infty$ algebra $(C, \mathfrak
m)$ \emph{with non-zero $\mathfrak m_0$-term} is associated to
each Lagrangian submanifold $L$ \cite{fooo00,fooo06}.
\subsection{$A_{\infty}$ algebra}
\label{subsec:objects}
In this subsection, we review the notion and construction of
filtered $A_{\infty}$ algebra associated to a Lagrangian
submanifold.
In order to make the construction consistent to one
in the last section, where $\Lambda(L_0,L_1;\ell_{01})$
is used for the coefficient ring rather than the universal
Novikov ring, we rewrite them
using smaller Novikov ring $\Lambda(L)$ which we define below.
Let $L$ be a relatively spin Lagrangian submanifold.
We have a homomorphism
$$
(E,\mu) : H_2(M,L;\Z) \to \R \times \Z
$$
where
$
E(\beta) = \beta \cap [\omega]
$
and $\mu$ is the Maslov index homomorphism.
We put $g\sim g'$ for $g,g' \in H_2(M,L;:\Z)$ if
$E(g) = E(g')$ and $\mu(g) = \mu(g')$.
We write $\Pi(L)$ the quotient with respect to this equivalence
relation. It is a subgroup of $\R\times \Z$.
We define
\beastar
\Lambda(L)
= \Big\{ \sum c_g[g] & \Big|& g \in \Pi(L), c_g \in R,
E(g) \ge 0, \\
&{}& \forall \, E_0 \, \# \{g \mid c_g \ne 0, E(g) \le E_0\} < \infty \Big \}
\eeastar
There exists an embedding
$
\Lambda(L) \to \Lambda_{0,nov}
$,
defined by $[g] \mapsto e^{\mu(g)/2}T^{E(g)}$.
\par
Let $\overline C$ be a graded $R$-module and $CF = \overline C
\widehat{\otimes}_R \Lambda(L)$. Here and hereafter we use symbol $CF$ for the modules over
$\Lambda(L)$ or $\Lambda(L_0,L_1)$ and
$C$ for the modules over the universal Novikov ring.
\par
We denote by $CF[1]$
its suspension defined by $CF[1]^k = CF^{k+1}$.
We denote by $\deg(x)=|x|$ the degree of $x \in C$ before the shift and
$\deg'(x)=|x|'$ that after the degree shifting, i.e., $|x|' = |x| - 1$.
Define the {\it bar complex} $B(CF[1])$ by
$$
B_k(CF[1]) = (CF[1])^{k\otimes}, \quad B(CF[1]) =
\bigoplus_{k=0}^\infty B_k(CF[1]).
$$
Here $B_0(CF[1]) = R$ by definition.
The tensor product is taken over $\Lambda(L)$.
We provide the degree of
elements of $B(CF[1])$ by the rule
\begin{equation}\label{eq:degonBC[1]}
|x_1 \otimes \cdots \otimes x_k|': = \sum_{i=1}^k |x_i|' = \sum_{i =1}^k|x_i|
-k
\end{equation}
where $|\cdot|'$ is the shifted degree. The ring $B(CF[1])$
has the structure of {\it graded coalgebra}.
\begin{defn} The structure of {\it strict filtered $A_\infty$ algebra}
over $\Lambda(L)$ is a
sequence of $\Lambda(L)$ module homomorphisms
$$
\mathfrak m_k: B_k(CF[1]) \to CF[1], \quad k = 1, 2, \cdots,
$$
of degree +1 such that the coderivation
$d = \sum_{k=1}^\infty \widehat{\mathfrak m}_k$
satisfies $d d= 0$, which
is called the \emph{$A_\infty$-relation}.
Here we denote by $\widehat{\mathfrak
m}_k: B(CF[1]) \to B(CF[1])$ the unique extension of $\mathfrak m_k$
as a coderivation on $B(CF[1])$. A \emph{filtered $A_\infty$ algebra}
is an $A_\infty$ algebra with a filtration for which $\mathfrak m_k$ are
continuous with respect to the induce non-Archimedean topology.
\end{defn}
In particular, we have $\mathfrak m_1
\mathfrak m_1 = 0$ and so it defines a complex $(CF,\mathfrak m_1)$. We
define the $\mathfrak m_1$-cohomology by
\begin{equation}\label{eq:m1cohom}
H(CF,\mathfrak m_1) = \mbox{Ker}\,\mathfrak m_1/\mbox{Im}\,\mathfrak m_1.
\end{equation}
A filtered {\it $A_\infty$ algebra} is defined in the same way, except
that it also includes
$$
\mathfrak m_0: R \to B(CF[1]).
$$
The first two terms of the $A_\infty$ relation for a
$A_\infty$ algebra are given as
\begin{eqnarray}
\mathfrak m_1(\mathfrak m_0(1)) & = & 0 \label{eq:m1m0=0} \\
\mathfrak m_1\mathfrak m_1 (x) + (-1)^{|x|'}\mathfrak m_2(x, \mathfrak m_0(1)) +
\mathfrak m_2(\mathfrak m_0(1), x) & = & 0. \label{eq:m0m1}
\end{eqnarray}
In particular, for the case $\frak m_0(1)$ is nonzero, $\mathfrak
m_1$ will not necessarily satisfy the boundary property, i.e., $\mathfrak m_1\mathfrak m_1
\neq 0$ in general.
\begin{rem}
Here we use the Novikov ring $\Lambda(L)$. In \cite{fooo06} we defined a filtered $A_{\infty}$
algebra over the universal Novikov ring $\Lambda_{0,nov}$.
A filtered $A_{\infty}$ algebra over $\Lambda(L)$ induces one over
$\Lambda_{0,nov}$ in an obvious way.
On the other hand, an appropriate gap condition is needed for
a filtered $A_{\infty}$ algebra over $\Lambda_{0,nov}$ to
induce one over $\Lambda(L)$.
\end{rem}

\par
We now describe the $A_\infty$ operators $\mathfrak m_k$
in the context of $A_\infty$ algebra of Lagrangian submanifolds.
For a given compatible almost complex structure $J$, consider the moduli
space of stable maps of genus zero
$$
\CM_{k+1}(\beta;L) =\{ (w, (z_0,z_1, \cdots,z_k)) \mid
\overline \partial w = 0, \, z_i \in \partial D^2, \, [w] = \beta
\, \mbox{in }\, \pi_2(M,L) \}/\sim
$$
where $\sim$ is the conformal reparameterization of the disc $D^2$.
We require that $z_0, \cdots, z_k$ respects counter clockwise cyclic order of
$S^1$.
(We wrote this moduli space $\CM^{\text{\rm main}}_{k+1}(\beta;L)$
in \cite{fooo08}. The symbol `main' indicates the compatibility
of $z_0, \cdots, z_k$, with counter clockwise cyclic order. We omit this
symbol in this paper since we always assume it.)
\par
$\CM_{k+1}(\beta;L)$ has a Kuranishi structure and its
dimension is given by
\begin{equation}\label{eq:dim}
n+ \mu(\beta) - 3 + (k+1) = n+\mu(\beta) + k-2.
\end{equation}
Now let
$
[P_1,f_1], \cdots,[P_k,f_k] \in C_*(L;\Q)
$
be $k$ smooth singular simplices of $L$.
(Here we denote by $C(L;\Q)$
a \emph{suitably chosen countably generated} cochain
complex of smooth singular chains of $L$.)
We put the cohomological grading
$\mbox{deg} P_i = n - \dim P_i$ and consider the fiber product
$$
ev_0: \CM_{k+1}(\beta;L) \times_{(ev_1, \cdots, ev_k)}(P_1 \times
\cdots \times P_k) \to L.
$$
A simple calculation shows that the expected dimension of this chain is given by
$
n + \mu(\beta) - 2 + \sum_{j=1}^k(\dim P_j + 1- n)
$
or equivalently we have the degree
$$
\mbox{deg}\left[\CM_{k+1}(\beta;L) \times_{(ev_1, \cdots, ev_k)}(P_1
\times \cdots\times P_k),
ev_0\right] = \sum_{j=1}^n(\mbox{deg} P_j -1) + 2- \mu(\beta).
$$
For each given $\beta \in \pi_2(M,L)$ and $k = 0, \cdots$, we
define $\frak m_{1,0}(P) = \pm \partial P$ and
\be\label{eq:mkbeta}
\aligned
\mathfrak m_{k,\beta}(P_1, \cdots, P_k)
&= \left[\CM_{k+1}(\beta;L) \times_{(ev_1, \cdots, ev_k)}(P_1 \times \cdots \times P_k),
ev_0\right] \\
& \in C(L;\Q)
\endaligned
\ee
(More precisely we regard the right hand side of (\ref{eq:mkbeta})
as a smooth singular chain by taking appropriate multi-valued
perturbation (multisection) and choosing a simplicial decomposition of its zero set.)
\par
We put
$$
CF(L) = C(L;\Q) \,\,\widehat{\otimes}_{\Q}\,\, \Lambda(L).
$$
We define
$
\mathfrak m_k : B_kCF(L)[1] \to B_kCF[1]
$
by
$$
\mathfrak m_k= \sum_{\beta \in \pi_2(M,L)} \mathfrak m_{k,\beta} \otimes [\beta].
$$
\par
Then it follows that the map
$
\mathfrak m_k : B_kCF(L)[1] \to CF(L)[1]
$
is well-defined, has degree 1 and continuous with respect to non-Archimedean topology.
We extend $\mathfrak m_k$ as a coderivation
$\widehat{\mathfrak m}_k: BCF[1] \to BCF[1]$
where $BCF(L)[1]$ is the completion of the direct sum $\oplus_{k=0}^\infty
B_kCF(L)[1]$ where $B_kCF(L)[1]$ itself is the completion of $CF(L)[1]^{\otimes k}$.
$BCF(L)[1]$ has a natural filtration defined similarly as Definition \ref{efilt}.
Finally we take the sum
$$
\widehat d = \sum_{k=0}^\infty \widehat{\frak m}_k : BCF(L)[1] \to BCF(L)[1].
$$
We then have the following coboundary property:
\begin{thm}\label{algebra} Let $L$ be an arbitrary compact relatively
spin Lagrangian submanifold of an arbitrary tame symplectic manifold
$(M,\omega)$. The coderivation $\widehat d$ is a continuous map that
satisfies the $A_\infty$ relation $\widehat d \widehat d = 0$, and so
$(CF(L),\mathfrak m)$ is a filtered $A_\infty$ algebra over $\Lambda(L)$.
\end{thm}
We put
$$
C(L;\Lambda_{0,nov}) = CF(L) \,\widehat{\otimes}_{\Lambda(L)} \,\, \Lambda_{0,nov}
$$
on which a filtered $A_{\infty}$ structure on $C(L;\Lambda_{0,nov})$
(over the ring $\Lambda_{0,nov}$) is induced.
This is the filtered $A_{\infty}$ structure given in Theorem A \cite{fooo06}.
The proof is the same as that of Theorem A \cite{fooo06}.
\par
In the presence of $\mathfrak m_0$,
$\widehat{\mathfrak m}_1 \widehat{\mathfrak m}_1 = 0$ no longer holds
in general. This leads to consider deforming Floer's original definition by a bounding
cochain of the obstruction cycle arising from bubbling-off discs.
One can always deform the given (filtered) $A_\infty$ algebra $(CF(L),\mathfrak m)$ by
an element $b \in CF(L)[1]^0$ by re-defining the $A_\infty$ operators as
$$
\mathfrak m_k^b(x_1,\cdots, x_k) = \mathfrak m(e^b,x_1, e^b,x_2, e^b,x_3, \cdots,
x_k,e^b)
$$
and taking the sum $\widehat d^b = \sum_{k=0}^\infty \widehat{\mathfrak m}_k^b$.
This defines a new filtered $A_\infty$ algebra in general. Here we
simplify notations by writing
$$
e^b = 1 + b + b\otimes b + \cdots + b \otimes \cdots \otimes b +\cdots.
$$
Note that each summand in this infinite sum has degree 0 in $CF(L)[1]$ and
converges in the non-Archimedean topology if $b$ has positive
valuation, i.e., $v(b) > 0$.
(See Section \ref{subsec:novikov} for the definition of $v$.)
\begin{prop} For the $A_\infty$ algebra $(CF(L),\mathfrak m_k^b)$,
$\mathfrak m_0^b = 0$ if and only if $b$ satisfies
\begin{equation}\label{eq:MC}
\sum_{k=0}^\infty\mathfrak m_k(b,\cdots, b) = 0.
\end{equation}
This equation is a version of \emph{Maurer-Cartan equation}
for the filtered $A_\infty$ algebra.
\end{prop}

\begin{defn}\label{boundchain}
Let $(CF(L),\mathfrak m)$ be a filtered $A_\infty$ algebra in general and
$BCF(L)[1]$ be its bar complex. An element $b \in CF(L)[1]^0 = CF(L)^1$ is called
a \emph{bounding cochain} if it satisfies the equation (\ref{eq:MC})
and $v(b) > 0$. We denote by $\widetilde \CM(L;\Lambda(L))$ the set of bounding
cochains.
\end{defn}

In general a given $A_\infty$ algebra may or may not have a solution
to (\ref{eq:MC}). In our case we define:

\begin{defn}\label{unobstructed}
A filtered $A_\infty$ algebra $(CF(L),\mathfrak m)$ is called \emph{unobstructed over}
$\Lambda(L)$ if the equation
(\ref{eq:MC}) has a solution $b \in CF(L)[1]^0 = CF(L)^1$ with $v(b) > 0$.
\end{defn}
One can define the notion of homotopy equivalence between two
bounding cochains and et al as described in Chapter 4 \cite{fooo06}. We denote
by $ \CM(L;\Lambda(L))$ the set of equivalence classes of bounding cochains of
$L$.
\begin{rem}
In Definition \ref{boundchain} above we consider
bounding cochain contained in
$CF(L) \subset C(L;\Lambda_0)$ only.
This is the reason why we write $\widetilde \CM(L;\Lambda(L))$
in place of $\widetilde \CM(L)$. (The latter is used in \cite{fooo06}.)
\end{rem}
\subsection{$A_{\infty}$ bimodule}
\label{subsec:morphisms}
Suppose we are in Situation \ref{pairdeta}. Once the $A_\infty$
algebra is attached to each Lagrangian submanifold $L$, we then
construct a structure of filtered \emph{$A_\infty$ bimodule} on the
module $CF((L_1,\gamma_1), (L_0,\gamma_0))$, which was introduced in
Section \ref{chaincomplex} as follows. This filtered $A_\infty$
bimodule structure is by definition is a family of operators
$$
\aligned
\mathfrak n_{k_1,k_0}: B_{k_1}(CF(L_1)[1])\,\,
\widehat{\otimes}_{\Lambda(L_1)} \,\, CF((L_1,\gamma_1),
(L_0,\gamma_0))
\,\,&\widehat{\otimes}_{\Lambda(L_0)} \,\, B_{k_0}(CF(L')[1]) \\
&\to CF((L_1,\gamma_1),
(L_0,\gamma_0))
\endaligned$$
for $k_0,k_1\ge 0$.
Here the left hand side is defined as follows:
It is easy to see that there are embeddings
$\Lambda(L_0) \to \Lambda(L_0,L_1;\ell_{01})$,
$\Lambda(L_1) \to \Lambda(L_0,L_1;\ell_{01})$.
Therefore a $ \Lambda(L_0,L_1;\ell_{01})$ module
$CF((L_1,\gamma_1),
(L_0,\gamma_0))$ can be regarded both as
$\Lambda(L_0)$ module and $\Lambda(L_1)$ module.
Hence we can take tensor product in the left hand side.
($\widehat{\otimes}_{\Lambda(L_i)}$ is the
completion of this algebraic tensor product.)
The left hand side then becomes a $ \Lambda(L_0,L_1;\ell_{01})$ module,
since the rings involved are all commutative.
\par
We briefly describe the definition of $\mathfrak n_{k_1,k_0}$.
A typical element of the tensor product
$$
B_{k_1}(CF(L_1)[1]) \widehat{\otimes}_{\Lambda(L_1)} \,\, CF((L_1,\gamma_1),
(L_0,\gamma_0))
\,\,\widehat{\otimes}_{\Lambda(L_0)} \,\, B_{k_0}(CF(L_0)[1])
$$
has the form
$$
P_{1,1} \otimes \cdots, \otimes P_{1,k_1} \otimes [p,w] \otimes
P_{0,1} \otimes \cdots \otimes P_{0,k_0}
$$
with $p \in L_0 \cap L_1$ being an admissible intersection point. Then the image $\mathfrak n_{k_0,k_1}$ thereof is given by
$$
\sum_{q, B}T^{\omega(B)}e^{\mu(B)/2}\# \left(\CM(p,q;B;P_{1,1},\cdots,P_{1,k_1};
P_{0,0},\cdots,P_{0,k_0})\right) [q,B\#w].
$$
Here $B$ denotes homotopy class of Floer trajectories connecting $p$ and $q$,
the summation is taken over all $[q,B]$ with
$$
\dim \CM(p,q;B;P_{1,1},\cdots,P_{1,k_1};
P_{0,1},\cdots,P_{0,k_0}) = 0,
$$
and $\# \left(\CM(p,q;B;P_{1,1},\cdots,P_{1,k_1};P_{0,1},\cdots,P_{0,k_0})\right)$ is
the `number' of elements in
the `zero' dimensional moduli space $\CM(p,q;B;P_{1,1},\cdots,P_{1,k_1};
P_{0,1},\cdots,P_{0,k_0})$. Here the moduli space $\CM(p,q;B;P_{1,1},\cdots,P_{1,k_1};
P_{0,1},\cdots,P_{0,k_0})$ is the Floer moduli space
$
\CM(p,q;B)
$
cut-down by intersecting with the given chains $P_{1,i} \subset L_1$
and $P_{0,j} \subset L_0$. (See Section 3.7 \cite{fooo08}.)
An orientation on this moduli space is induced by $o_{[p,w]}$, $o_{[q,w']}$,
which we obtained by Theorem \ref{thm:fooo-ori}.
\begin{thm}\label{thm;bimodule}
Let $((L_0,\gamma_0),
(L_1,\gamma_1))$ be a pair of anchored
Lagrangian submanifolds.
Then the family $\{\mathfrak n_{k_1,k_0}\}$ defines
a left $(CF(L_1),\mathfrak m)$ and right $(CF(L_0),\mathfrak m)$ filtered
$A_\infty$-bimodule structure on $CF((L_1,\gamma_1),
(L_0,\gamma_0))$.
\end{thm}
See Section 3.7 \cite{fooo08} for the definition of filtered $A_{\infty}$ bimodules.
(In \cite{fooo08} the case of universal Novikov ring as a coefficient is considered.
It is easy to modify to our case of $\Lambda(L_0,L_1)$ coefficient.)
The proof of Theorem \ref{thm;bimodule} is the same as that of Theorem 3.7.21 \cite{fooo08}.
\par
In the case where both $L_0, \, L_1$ are
unobstructed, we can carry out this
deformation of $\mathfrak n$ using bounding cochains $b_0$
and $b_1$ of $CF(L_0)$ and $CF(L_1)$ respectively, in
a way similar to $\frak m^b$. Namely we define
$\delta_{b_1,b_0} : CF((L_1,\gamma_1),
(L_0,\gamma_0)) \to CF((L_1,\gamma_1),
(L_0,\gamma_0))$ by
$$
\delta_{b_1,b_0}(x) =
\sum_{k_1,k_0} \frak n_{k_1,k_0} (b_1^{\otimes k_1} \otimes
x \otimes b_0^{\otimes k_0}) = \mathfrak{\widehat n}(e^{b_1},x,e^{b_0}).
$$
We can generalize the story to the case where $L_0$ has clean intersection
with $L_1$, especially to the case $L_0=L_1$.
In the case $L_0=L_1$ we have
$\frak n_{k_1,k_0} = \frak m_{k_0+k_1+1}$. So
in this case, we have
$
\delta_{b_1,b_0}(x) = \frak m(e^{b_1},x,e^{b_0}).
$
\par
We define Floer cohomology of the pair $(L_0,\gamma_0,\lambda_0)$, $(L_1,\gamma_1,\lambda_1)$
by
$$
HF((L_1,\gamma_1,b_1),(L_0,\gamma_0,b_0)) = \operatorname{Ker}\delta_{b_1,b_0}/\operatorname{Im}
\delta_{b_1,b_0}.
$$
This is a module over $\Lambda(L_0,L_1;\ell_{01})$.
\begin{thm}
$HF((L_1,\gamma_1,b_1),(L_0,\gamma_0,b_0)) \otimes_{\Lambda(L_0,L_1)} \Lambda_{nov}$
is invariant under the
Hamiltonian isotopies of $L_0$ and $L_1$ and under the
homotopy of bounding cochains $b_0, \, b_1$.
\end{thm}
The proof is the same as the proof of Theorem 4.1.5 \cite{fooo08}
and so omitted.
\subsection{Products}\label{subsecprod}
Let $\frak L = (L_0, L_1, \cdots, L_k)$ be a
chain of compact Lagrangian submanifolds in $(M,\omega)$
that intersect pairwise transversely without triple intersections.
\par
Let $\vec z = (z_{0k},z_{k(k-1)},\cdots,z_{10})$ be a set of distinct points on $\partial D^2
= \{ z\in \C \mid \vert z\vert = 1\}$. We assume that
they respect the counter-clockwise cyclic order of $\partial D^2$.
The group $PSL(2;\R)\cong \operatorname{Aut}(D^2)$ acts on the set
in an obvious way. We denote by $\mathcal M^{\text{main},\circ}_{k+1}$ be
the set of $PSL(2;\R)$-orbits of $(D^2,\vec z)$.
\par
In this subsection, we consider only the case $k \geq 2$
since the case $k=1$ is already discussed in the last subsection.
In this case there is no automorphism on the domain $(D^2, \vec z)$, i.e.,
$PSL(2;\R)$ acts freely on the set of such $(D^2, \vec z)$'s.
\par
Let
$p_{j(j-1)} \in L_j \cap L_{j-1}$
($j = 0,\cdots k$), be a set of intersection points.
\par
We consider the pair $(w;\vec z)$ where $w: D^2 \to M$ is a
pseudo-holomorphic map that satisfies the boundary condition
\begin{subequations}\label{54.15}
\begin{eqnarray}
&w(\overline{z_{j(j-1)}z_{(j+1)j}}) \subset L_j, \label{54.15.1} \\
&w(z_{(j+1)j}) = p_{(j+1)j}\in L_j \cap L_{j+1}. \label{54.15.2}
\end{eqnarray}
\end{subequations}
We denote by $\widetilde{\CM}^{\circ}(\frak L, \vec p)$
the set of such
$((D^2,\vec z),w)$.
\par
We identify two elements $((D^2,\vec z),w)$, $((D^2,\vec z'),w')$
if there exists $\psi \in PSL(2;\R)$ such that
$w \circ \psi = w'$ and $\psi(z'_{j(j-1)}) = z_{j(j-1)}$.
Let ${\CM}^{\circ}(\frak L, \vec p)$ be the set of equivalence classes.
We compactify it by including the configurations with disc or sphere bubbles
attached, and denote it by ${\CM}(\frak L, \vec p)$.
Its element is denoted by $((\Sigma,\vec z),w)$ where
$\Sigma$ is a genus zero bordered Riemann surface with one boundary
components, $\vec z$ are boundary marked points, and
$w : (\Sigma,\partial\Sigma) \to (M,L)$ is a bordered stable map.
\par
We can decompose $\CM(\frak L, \vec p)$ according to the homotopy
class $B \in \pi_2(\frak L,\vec p)$ of continuous maps satisfying
\eqref{54.15.1}, \eqref{54.15.2} into the union
$$
\CM(\frak L, \vec p) = \bigcup_{B \in \pi_2(\frak L;\vec p)}
\CM(\frak L, \vec p;B).
$$
\par
In the case we fix an anchor $\gamma_i$ to each of $L_i$ and put $\CE =
((L_0,\gamma_0),\cdots,(L_k,\gamma_k))$, we consider only
admissible classes $B$ and put
\par
$$
\CM(\CE, \vec p) = \bigcup_{B \in \pi_2^{ad}(\CE;\vec p)}
\CM(\CE, \vec p;B).
$$
\begin{thm}\label{58.21} Let $\frak L
= (L_0,\cdots,L_k)$ be a chain of
Lagrangian submanifolds and
$B \in \pi_2(\frak L;\vec p)$.
Then $\CM(\frak L, \vec p;B)$ has an oriented Kuranishi structure
(with boundary and corners). Its (virtual) dimension satisfies
\begin{equation}\label{dimensionformula}
\dim \CM(\frak L, \vec p;B) = \mu(\frak L,\vec p;B) + n + k-2,
\end{equation}
where $\mu(\frak L,\vec p;B)$ is the polygonal Maslov index of $B$.
\end{thm}
\begin{proof}
We consider the operator $\delbar_{w;\frak L}$ in (\ref{delbawl}).
It is easy to see that
\begin{equation}\label{dimwfrak}
\text{Index}\,\, \delbar_{w;\frak L} + k -2= \dim \CM(\frak L, \vec p;B).
\end{equation}
In fact $k-2$ in the left hand side is the dimension of
$\mathcal M^{\text{main},\circ}_{k+1}/PSL(2;\R)$.
\par
We next consider the Fredholm operator
(\ref{summedupindex}).
By (\ref{dimwfrak}), Lemma \ref{thm:poly} and index sum formula, we have
\begin{equation}
\text{Index}\,\, \delbar_{w;\frak L} - \mu(\frak L,\vec p;B)
= \text{Index of (\ref{summedupindex})}.
\end{equation}
We remark that the operator (\ref{summedupindex}) is a
Cauchy-Riemann operator of the trivial $\C^n$ bundle on $D^2$ with
boundary condition determined by a certain loop in
$Lag(\C^n,\omega)$. By construction it is easy to see that this loop
is homotopic to a constant loop. Therefore, the index of
(\ref{summedupindex}) is $n$. Theorem \ref{58.21} follows.
\end{proof}
We next take graded anchors $(\gamma_i,\lambda_i)$
to each $L_i$ and fix the data as in Situation \ref{pairdeta}.
We assume that $B$ is admissible and write
$
B = [w^-_{01}]\#[w^-_{12}] \# \cdots \# [w^-_{k0}]
$ as in Definition \ref{classB}.
We put $w^+_{(i+1)i}(s,t) = w^-_{i(i+1)}(1-s,t)$ as in (\ref{wplus}).
We also put $w^+_{k0}(s,t) = w^+_{0k}(s,1-t)$.
($[w^+_{k0}] \in \pi_1(\ell_{k0};p_{k0})$.)
We also put $\lambda_{k0}(t) = \lambda_{0k}(1-t)$.
\begin{lem}\label{dimanddeg}
If $\dim \CM(\frak L, \vec p;B) = 0$, we have
\begin{equation}\label{misdeg1}
(\mu([p_{k0},w^+_{k0}];\lambda_{0k}) - 1)
=
1 + \sum_{i=1}^k (\mu([p_{i(i-1)},w^+_{i(i-1)}];\lambda_{(i-1)i}) - 1).
\end{equation}
\end{lem}
\begin{proof}
Lemma \ref{thm:poly} and Theorem \ref{58.21} implies
$$
\sum_{i=0}^{k}
\mu([p_{(i+1)i},w^+_{(i+1)i}];\lambda_{i(i+1)})
= n + k -2
$$
in the case $\dim \CM(\frak L, \vec p;B) = 0$.
By Lemma \ref{degPDlem} we have
$$
\mu([p_{0k},w^+_{0k}];\lambda_{k0})=
- \mu([p_{k0},w^+_{k0}];\lambda_{0k}) + n.
$$
Substituting this into the above identity and rearranging the identity, we obtain
the lemma.
\end{proof}
Using the case $\dim \CM(\frak L, \vec p;B) = 0$, we define the $k$-linear operator
$$
\frak m_k:
CF((L_k,\gamma_k),(L_{k-1},\gamma_{k-1}))
\otimes \ldots \otimes
CF((L_1,\gamma_1),(L_{0},\gamma_{0}))\to CF((L_k,\gamma_k),(L_0,\gamma_0))
$$
as follows:
\begin{equation}\label{catAinifwob}
\aligned
\frak m_{k}([p_{k(k-1)},w^+_{k(k-1)}],& [p_{(k-1)(k-2)},w^+_{(k-1)(k-2)}],\cdots, [p_{10},w^+_{10}])) \\
&= \sum \#(\CM_{k+1}(\frak L;\vec p;B)) \, [p_{k0},w^+_{k0}]).
\endaligned\end{equation}
Here the sum is over the basis $[p_{k0},w^+_{k0}]$ of
$CF((L_k,\gamma_k),(L_0,\gamma_0))$, where
$\vec p = (p_{0k},p_{k(k-1)},\cdots,p_{10})$,
$B$ is as in Definition \ref{classB}, and
$w^+_{(i+1)i}(s,t) = w^-_{i(i+1)}(1-s,t)$.
\par
The formula (\ref{misdeg1}) implies that $\frak m_k$ above has
degree one.
\par
In general the operator $\frak m_k$ above does {\it not} satisfy the
$A_{\infty}$ relation by the same reason as that of the case of
boundary operators (see Section \ref{chaincomplex}). We need to use
bounding cochains $b_i$ of $L_i$ to deform $\frak m_k$ in the same
way as the case of $A_\infty$-bimodules (Subsection
\ref{subsec:morphisms}), whose explanation is now in order.
\par
Let $m_0,\cdots,m_k \in \Z_{\ge 0}$ and $\CM_{m_0,\cdots,m_k}(\frak L,
\vec p;B)$ be the moduli space obtained from the set of $((D^2,\vec
z),(\vec z^{(0)},\cdots,\vec z^{(k)}),w))$ by taking the quotient by
$PSL(2,\R)$-action and then by taking the stable map
compactification as before. Here $ z^{(i)} =
(z^{(i)}_1,\cdots,z^{(i)}_{k_i})$ and $z^{(i)}_{j} \in
\overline{z_{(i+1)i}z_{i(i-1)}}$ such that
$z_{(i+1)i},z^{(i)}_1,\cdots,z^{(i)}_{k_i}, z_{i(i-1)}$ respects the counter
clockwise cyclic ordering.
$$
((D^2,\vec z),(\vec z^{(0)},\cdots,\vec z^{(k)}),w))
\mapsto (w(z^{(0)}_1),\cdots,w(z^{(k)}_{m_k}))
$$
induces an evaluation map:
$$
ev=(ev^{(0)},\cdots,ev^{(k)}): \CM_{m_0,\cdots,m_k}(\frak L, \vec p;B)
\to \prod_{i=0}^k L_i^{m_i}.
$$
Let $P^{(i)}_j$ be smooth singular chains of $L_i$ and put
$$
\vec P^{(i)} = (P^{(i)}_1,\cdots,P^{(i)}_{m_i}),
\qquad
\vec{\vec P} = (\vec P^{(0)},\cdots,\vec P^{(k)})
$$
We then take the fiber product to obtain:
$$
\CM_{m_0,\cdots,m_k}(\frak L, \vec p;\vec{\vec P};B) =
\CM_{m_0,\cdots,m_k}(\frak L, \vec p;B) \times_{ev} \vec{\vec P}.
$$
We use this to define
$$
\aligned &\frak m_{k;m_0,\cdots,m_k} : B_{m_k}(CF(L_k)) \otimes
CF((L_k,\gamma_k),(L_{k-1},\gamma_{k-1}))
\otimes \cdots\\
&\quad\otimes CF((L_{1},\gamma_{1}),(L_0,\gamma_{0}))
\otimes B_{m_0}(CF(L_0))
\to CF((L_{k},\gamma_{k}),(L_0,\gamma_{0}))
\endaligned
$$
by
$$
\aligned
\frak m_{k;m_0,\cdots,m_k}
(\vec P^{(k)},[p_{k(k-1)},w^+_{k(k-1)}],&\cdots,[p_{10},w^+_{10}],
\vec P^{(0)})
\\
& = \sum \#(\CM_{k+1}(\frak L;\vec p;\vec{\vec P};B)) \, [p_{k0},w_{k0}].
\endaligned$$
Finally for each given $b_i \in CF(L_i)[1]^0$ ($b_i \equiv 0 \mod \Lambda_+$),
$\vec b =(b_0,\cdots,b_k)$, and $x_i \in CF((L_i,\gamma_i),(L_{i-1},\gamma_{i-1}))$, we put
\begin{equation}\label{mkcorrected}
\frak m_k^{\vec b}(x_k,\cdots,x_1) = \sum_{m_0,\cdots,m_k} \frak
m_{k;m_0,\cdots,m_k} (b_k^{m_k},x_k,b_{k-1}^{m_{k-1}},\cdots,x_1,b_0^{m_0}).
\end{equation}
\begin{thm}
If $b_i$ satisfies the Maurer-Cartan equation $(\ref{eq:MC})$ then
$\frak m_k^{\vec b}$ in $(\ref{mkcorrected})$ satisfies the
$A_{\infty}$ relation
\begin{equation}\label{Ainftyrel}
\sum_{k_1,k_2,i}
(-1)^* \frak m_{k_1}(x_k,\cdots,\frak m_{k_2}(x_{k-i-1},\cdots,x_{k-i-k_2}),\cdots,x_1) = 0
\end{equation}
where we take sum over $k_1+k_2=k+1$, $i=-1,\cdots,k-k_2$.
(We write $\frak m_k$ in place of $\frak m^{\vec b}_k$ in $(\ref{Ainftyrel})$.)
The sign $*$ is
$
* = i + \deg x_k +\cdots + \deg x_{k-i}.
$
\end{thm}
The non-anchored version is proved in Theorem 4.17 \cite{Fuk02II}.
In order to translate it to the anchored version we only need to
show the following.
\begin{lem}\label{lem:split}
Let $\CE = ((L_0,\gamma_0),\cdots,(L_i,\gamma_i),\cdots,
(L_j,\gamma_j),\cdots,(L_k,\gamma_k)), \, \vec p = (p_{k(k-1)},\cdots, p_{j(j-1)},\cdots, p_{i(i-1)},\cdots
p_{10})$ and $B \in \pi_2^{ad}(\CE;\vec p)$ be admissible.
Suppose that the sequence $u_i \in \CM(\frak L, \vec p;B)$ converges
to an element in the product $\CM(\frak L', \vec p\,';B_1) \times
\CM({\frak L^{\prime\prime}},\vec p^{\,\prime\prime};B_2)$ where
$$\frak L' = (L_0,\cdots,L_i,L_j,\cdots,L_k), \quad
\frak L^{\prime\prime} = (L_i,L_{i+1},\cdots,L_j)$$
and
$$\vec p\,' = (p_{0k},p_{k(k-1)},\cdots,p_{(j+1)j},p_{ji},p_{i(i-1)},\cdots,p_{10}),
\quad
\vec p^{\,\prime\prime} = (p_{ij},p_{j(j-1)}\cdots,p_{(i+1)i})
$$
for some $p_{ij}= p_{ji} \in L_i \cap L_j$.
\par
Then $B_1$, $B_2$ are $\CE'$, $\CE^{\prime\prime}$
admissible, respectively.
\end{lem}
\begin{proof}
For simplicity of notations, we only consider the case $k=3$, $i=1$,
$j=3$. Let $u_a \in \CM(\frak L,\vec p;B)$ which converges to
$u_{\infty} = (u_{\infty,1},u_{\infty,2})$ where
$$\aligned
&u_{\infty,1} \in \CM((L_0,L_1,L_3),(p_{03},p_{31},p_{10});B_1),
\\
&u_{\infty,2} \in \CM((L_1,L_2,L_3),(p_{13},p_{32},p_{21});B_2),
\endaligned
$$
and $p_{13} = p_{31}$. By definition of $\CE$-admissibility of $B$,
there exist homotopy classes $[w^-_{i(i+1)}] \in \pi_2(L_{i(i+1)};\ell_{i(i+1)})$
for $0 \leq i \leq 3$ such that
$B = [w^-_{01}\# w^-_{12} \# w^-_{23} \# w^-_{30}]$.
To prove the required admissibility of $B_1, \, B_2$,
we need to prove the existence of homotopy classes $[w^-_{13}] \in \pi_2(\ell_{13};p_{13})$
and $[w^-_{31}] \in \pi_2(\ell_{31};p_{31})$ such that
$$
[w^-_{01}\# w^-_{13} \# w^-_{30}] = B_1, \, [w^-_{12}\# w^-_{23} \# w^-_{31}] = B_2
$$
and $[w^-_{13}\# w^-_{31}] = [\widehat p_{13},p_{13}]$ where $\widehat p_{13}$ is
the constant map to $p_{13}$. In fact, we will select both
homotopy classes to be that of $\widehat p_{13}=\widehat p_{31}$.
\par
Exploiting $\CE$-admissibility of $B$, we can take the sequences of
points $x_a \in u_a(D^2) \subset M$, of
paths $\gamma_a : [0,1] \to M$ and $\gamma_{a,i} : [0,1] \to
u_a(D^2)$ such that
\beastar
\gamma_a(0) = y, \quad \gamma_a(1) = x_a, \\
\gamma_{a,i}(0) = x_a, \quad \gamma_{a,i}(1) \in L_i
\eeastar
and that
$\gamma_a*\gamma_{a,i}$ is homotopic to the given anchor $\gamma_i$.
\par
Deforming these choices further, we may use the convergence hypothesis to achieve the following
additional properties of $x_a, \, \gamma_a$ and $\gamma_{a,i}$'s:
\begin{enumerate}
\item $\lim_{a \to \infty} x_a = p_{13}$,
\item $\lim_{a\to\infty} \gamma_{a,i}(t) \equiv p_{13}$ for $i=1, \, 3$,
\item $\gamma_{a,i}$ converges to a path $\gamma_{\infty,i}$ as $a \to \infty$,
\item and $\gamma_a$ converges to a path $\gamma_{\infty}$ as $a \to \infty$.
\end{enumerate}
From this and by construction of $\gamma_i, \, \gamma_{a,i}$, it is easy to see that $B_1$ and
$B_2$ are $((L_0,\gamma_{\infty,0}),(L_1,\gamma_{\infty}),(L_3,\gamma_{\infty}))$ admissible
and $((L_1,\gamma_{\infty}),(L_2,\gamma_{2,\infty}),(L_3,\gamma_{\infty}))$ admissible respectively.
In fact, since $\gamma_{\infty}$ is homotopic to $\gamma_i$ for $i=1,3$,
$\gamma_{\infty,j}$ is homotopic to $\gamma_j$ for $j = 0,2$, we can express
$$
B_1 = [w^-_{01}\# \widehat p_{13} \# w^-_{30}], \,
B_2 = [w^-_{12}\# w^-_{23} \# \widehat p_{31}].
$$
This finishes the proof.
\end{proof}
We summarize the above discussion as follows:
\begin{thm}\label{anchoredAinfty}
We can associate an filtered $A_{\infty}$ category to a
symplectic manifold $(M,\omega)$ such that:
\begin{enumerate}
\item Its object is $((L,\gamma,\lambda),b,sp)$ where
$(L,\gamma,\lambda)$ is a graded anchored Lagrangian submanifold, $[b]
\in \CM(CF(L))$
is a bounding cochain and $sp$ is a spin structure of $L$.
\par
\item The set of morphisms is $CF((L_1,\gamma_1),(L_0,\gamma_0))$.
\par
\item $\frak m^{\vec b}_k$ are the operations defined in $(\ref{mkcorrected})$.
\end{enumerate}
\end{thm}
\begin{rem}
In Situation \ref{pairdeta}, beside the choices spelled out in $((L,\gamma,\lambda),b,sp)$,
the choice of orientations $o_p$ of $\text{Index}\, \overline{\partial}_{\lambda_p}$
is included. This choice in fact does not affect the module structure
$CF((L_1,\gamma_1),(L_0,\gamma_0))$ up to isomorphism:
if we take an alternative choice $o'_p$ at $p$, then all the
signs appearing in the operations $\frak m_k$ that involves $[p,w]$ for some $w$
will be reversed. Therefore $[p,w] \mapsto -[p,w]$ gives the required isomorphism.
\end{rem}
\begin{rem}
In \cite{Fuk02II}, the filtered $A_{\infty}$ category is defined over $\Lambda_{0,nov}$.
The situation of Theorem \ref{anchoredAinfty} is slightly
different in that $\Lambda(L_0,L_1;\ell_{01})$ or $\Lambda(L)$
are used as the coefficient rings and hence the coefficient rings vary depending on
the objects involved. It is easy to see that the notion of filtered $A_{\infty}$ category
can be generalized to this context.
\par
We can also change the coordinate ring to $\Lambda_{nov}$ by using
the map
$
[p,w] \mapsto T^{\int w^*\omega}e^{\mu(w)/2}\langle p \rangle
$
(Subsection 5.1.3 \cite{fooo08}).
The resulting filtered $A_{\infty}$ category is still different from the non-anchored
version in the case $M$ is not simply connected.
\end{rem}
\begin{rem}
In Theorem \ref{anchoredAinfty}, we assume that our Lagrangian submanifold $L$ is spin.
We can slightly modify the construction to accommodate the relatively spin case as follows:
We will construct the filtered $A_{\infty}$ category of $((M,\omega),st)$
for each choice of $st \in H^2(M;\Z_2)$.
Its objects consist of $((L,\gamma,\lambda),b,sp)$ where
$L$ satisfies $w_2(L) = i^*(st)$ and $(L,\gamma,\lambda),\, b$
are as before. Finally $sp$ is the
stable conjugacy class of relative spin structure of $L$.
(See Definition 8.1.5 \cite{fooo08} for its definition.)
In this way we obtain a filtered $A_{\infty}$ category.
The same remark applies to the non-anchored case.
\end{rem}
The operations $\frak m_k$ are compatible with the filtration.
Namely we have
\begin{prop}\label{filprod}
If $x_i \in F^{\lambda_i}CF((L_i,\gamma_i),(L_{i-1},\gamma_{i-1}))$,
then
$$
\frak m_k^{\vec b}(x_k,\cdots,x_1)
\in F^{\lambda}CF((L_k,\gamma_k),(L_0,\gamma_0))
$$
where
$
\lambda = \sum_{i=1}^{k} \lambda_i.
$
\end{prop}
Studying the behavior of filtration under the $A_\infty$ operations,
one can define (higher-order) spectral invariants of Lagrangian
Floer theory in a way similar to the one carried out in
\cite{oh:spectre}. Then Proposition \ref{filprod} implies a similar
estimate as Theorem I(4) \cite{oh:spectre}. This is a subject of
future study.
\section{Comparison between anchored and non-anchored versions}
\label{sec:relation}
The anchored Lagrangian Floer theory presented in this paper is somewhat different
from the one developed in \cite{fooo00, fooo06} (for one and two Lagrangian submanifolds)
\cite{Fuk02II} (for 3 or more Lagrangian submanifolds)
in several points. In this section we examine their relationship
and make some comments on some aspects of their applications.

We first point out the following obvious fact:
\begin{prop}
If $M$ is simply connected the anchored version of Floer homology is isomorphic
to non-anchored version together with all of its multiplicative structures.
\end{prop}
\par
We also remark that the way how we treat the orientation for the anchored version
in Section \ref{sec:orient} is actually the same as the one used for the non-anchored version
in \cite{fooo06} and \cite{Fuk02II}.
\subsection{Examples}
\label{subsec:exa}
We start with simple examples that illustrate some difference
between the two.

We consider the symplectic manifold
$(T^2,dx\wedge dy)$, where $T^2 = \R^2/\Z^2$ and $(x,y)$ is the standard coordinate
of $\R^2$.
Let $L_0 = \{[x,0] \mid x \in \R \}$, $L_1 = \{ [x,3x] \mid x\in \R\}$.
$$
L_0 \cap L_1 = \{[0,0],[1/3,0],[2/3,0]\}.
$$
Let $[1/2,0]$ be the base point and take anchors
$$
\gamma_0(t) = [(1-t)/2,0],
\quad
\gamma_1^0(t) = [(1-t)/2,0],
$$
of $L_0$ and $L_1$ respectively.
It is easy to see that $[0,0]$ is $((L_0,\gamma_0),(L_1,\gamma_1^0))$ admissible.
It is also easy to see by drawing pictures that $[1/3,0]$, $[2/3,0]$
are not $((L_0,\gamma_0),(L_1,\gamma_1^0))$ admissible.
The set of homotopy classes of anchors of $L_1$ is identified with $\Z$ and
$[k/3,0]$ is $((L_0,\gamma_0),(L_1,\gamma_1^{\ell}))$ admissible if $k\equiv \ell
\mod 3$.
Here anchor $\gamma_1^i$ is a concatenation of $\gamma_1^0$ and
$
t \mapsto [it/3,0]
$.
We also remark that $\Pi(L_0), \Pi(L_1), \Pi(L_0,L_1)$ are trivial.
Therefore $\Lambda(L_0) = \Lambda(L_1) = \Lambda(L_0,L_1) = \Q$.
Thus we have
$$
HF((L_1,\gamma_1^i),(L_0,\gamma_0)) \cong \Q
$$
for any choice of anchor $\gamma_1^i$.
\begin{rem}
In this subsection, we always take $0$ as the bonding cochain $b$ and
we omit it from the notation of Floer cohomology.
\end{rem}

On the other hand we have
$$
HF(L_1,L_0;\Lambda_{nov}) \cong \Lambda_{nov}^{\oplus 3}.
$$
It is easy to see that
$
\pi_0(\Omega(L_0,L_1))
$
consists of $3$ elements, which we denote $\ell_{01}^i$ $(i=0,1,2)$.
Moreover
$
[\overline{\gamma_0}*\gamma_1^j] = \ell_{01}^i
$
with $i \equiv j \mod 3$.
\par
Hence we have the decomposition
$$
HF(L_1,L_0;\Lambda_{nov})
\cong \bigoplus_{i=0}^2 HF(L_1,L_0;\ell_{01}^i;\Lambda_{nov}).
$$
This is the decomposition given in Remark 3.7.46 \cite{fooo08}.
(We note that we have the isomorphism
$$
HF(L_1,L_0;\ell_{01}^i;\Lambda_{nov})
\cong
HF((L_1,\gamma_1^i),(L_0,\gamma_0)) \otimes \Lambda_{nov}.)
$$
\par\medskip
We next consider the same $T^2$ and
$$
L_0 = \{ [0,y] \mid y \in \R \}, \quad L_1 = \{[x,0] \mid x \in \R\}.
$$
Then $L_0 \cap L_1$ consists of one point $[0,0]$.
Therefore
$$
HF((L_1,\gamma_1),(L_0,\gamma_0)) \cong \Q
$$
for any anchor $\gamma_0$ and $\gamma_1$.
In fact $\Omega(L_0,L_1)$ is connected in this case.
We next consider the third Lagrangian submanifold
$
L_2 = \{ [x,-x] \mid x \in \R\}.
$
It is also easy to see that
$$
HF(L_2,L_i;\Lambda_{nov}) \cong HF((L_2,\gamma_2),(L_i,\gamma_i))
\otimes \Lambda_{nov} \cong \Lambda_{nov},
$$
for $i=0,1$ and any anchor $\gamma_j$ of $L_j$.
\par
We take $[0,0]$ as base point and take anchors
$$
\gamma_0^k(t) = [kt,0], \quad \gamma_1^k(t) = [0,kt], \quad \gamma_2^k(t) = [kt/2,kt/2].
$$
of $L_i$ for $i=0,1,2$.
For each $k,\ell \in \Z$ and $i,j$ ($i,j \in \{0,1,2\}$, $i \ne j$) we have
$HF((L_i,\gamma_i^k),(L_j,\gamma_j^\ell)) \cong \Q$.
Let $x_{ij}^{k\ell} = [[0,0],w_{ij;k\ell}]$ be its canonical
generator.
Here $[w_{ij;k\ell}]$ represents the unique element of $\pi_2(\overline{\gamma_i^k}*\gamma_{j}^{\ell},[0,0])$.
\par
Let $x_{ij}=\langle[0,0]\rangle$ be also the canonical generator of the (non-anchored) Floer
homology $HF(L_i,L_j)$. (We refer readers to Section \ref{redbynonanchor} of
present paper and Subsection 5.1.3 \cite{fooo08} for the definition of $\langle[0,0]\rangle$.)
\par
Now the product $\frak m_2$ is described as follows:
\begin{prop}\label{prodcomp}
In the case of non-anchored version we have
\begin{equation}\label{theta}
\frak m_2(x_{21},x_{10}) = \left(\sum_{k\in \Z} T^{k^2/2} \right) x_{20}.
\end{equation}
In the anchored version we have
\begin{equation}\label{oneterm}
\frak m_2(x_{21}^{m\ell},x_{10}^{\ell k})
=
x_{20}^{mk}.
\end{equation}
\end{prop}
\begin{proof}
We first remark that
$\pi_2((L_0,L_1,L_2),(p_{02},p_{21},p_{10})) \cong \Z$.
Moreover each of the homotopy class is realized by
holomorphic disc uniquely. This implies (\ref{theta}).
\par
To prove (\ref{oneterm}).
it suffices to see that for each
$\gamma_0^k$, $\gamma_1^{\ell}$, $\gamma_2^m$ the set
\begin{equation}\label{3tuadmissible}
\pi^{adm}_2((L_0,\gamma_0^k),(L_1,\gamma_2^{\ell}),(L_2,\gamma_2^{m})),(p_{02},p_{21},p_{10}))
\end{equation}
of admissible class consists of one element.
We will prove it below.
\par
Let $B$ be an element of (\ref{3tuadmissible}).
We write
$
B = [w^-_{01}] \# [w^-_{12}] \# [w^-_{20}]
$
as in Definition \ref{classB}.
Let $\R^2 \to T^2$ be the universal covering.
We lift anchors $\gamma_0^k$, $\gamma_1^{\ell}$, $\gamma_2^m$
to $\widetilde\gamma_0^k$, $\widetilde\gamma_1^{\ell}$, $\widetilde\gamma_2^m$
such that
$\widetilde\gamma_0^k(0) = \widetilde\gamma_1^{\ell}(0) = \widetilde\gamma_2^m(0) = 0$.
\par
We then lift $w_{01}$ such that (a part of) its boundary is $\widetilde\gamma_0^k$
and $\widetilde\gamma_1^{\ell}$. We lift $w_{12}$ and $w_{20}$ in a similar
way. We thus obtain a lift $\widetilde w$ of $w$.
It is easy to see that the boundary of $\widetilde w(D^2)$ is contained in
$\widetilde L_0^k \cup \widetilde L_1^{\ell} \cup \widetilde L_2^m$ where
$$
\widetilde L_0^{k} = \{ (k,y) \mid y \in \R\}, \quad
\widetilde L_1^{\ell} = \{ (x,\ell) \mid x \in \R\}, \quad
\widetilde L_2^{m} = \{ (x,m-x) \mid x\in \R\}.
$$
Thus the admissible homotopy class of $B$ is unique.
\end{proof}

\begin{rem}
We remark that (\ref{theta}) is the formula appearing in Kontsevich \cite{konts:hms}
where the homological mirror symmetry proposal first appeared.
So it seems that the anchored version is not suitable for the
application to mirror symmetry, when $M$ is not simply connected.
On the other hand, the anchored version is more closely related to
the variational theoretical origin of Floer homology and so
seems more suitable to study spectral invariant for example.
\end{rem}

The above proof also implies the following:
\begin{lem}
If $B$ is $(L_0,\gamma_0^k),(L_1,\gamma_2^{\ell}),(L_2,\gamma_2^{m}))$
admissible then
\begin{equation}
B\cap \omega = \frac{(m-k-\ell)^2}{2}.
\end{equation}
\end{lem}
\begin{rem}
In the case of Lemma \ref{prodcomp}
we obtain the non-anchored version by summing up
anchored versions appropriately. In general the non-anchored version is
an appropriate sum of anchored versions.
However the way of summing up anchored versions to obtain the
non-anchored one does not look so simple to describe.
\end{rem}

\subsection{Relationship with the grading of Lagrangian submanifolds.}
\label{subsec:reldeg}
In \cite{Fuk02II} the first named author followed the method of
Seidel \cite{seidel:grading} (and Kontsevich) to define a grading of
Floer cohomology. In this section we discuss its relation to the
formulation of Section \ref{sec:grading}.
\par
We first briefly recall the notion of gradings in the sense of
\cite{seidel:grading}. Consider the tangent space $T_pM$ and let $Lag^+(T_pM)$
be the set of oriented Lagrangian subspaces.
The union $Lag^+(M): = \cup_{p \in M} Lag^+(T_pM)$ forms a
fiber bundle over $M$. If $L$ is an oriented Lagrangian submanifold,
the Gauss map $p \mapsto T_pL$ provides a canonical section
of the restriction $Lag^+(M)|_L \to L$. We denote the canonical section
by $\overline s_L$.
\par
We first consider the case $(M,\omega)$ with $c^1(M) = 0$.
\par
The fundamental group of $Lag^+(T_pM)$ is $\Z$.
The condition $c^1(M) = 0$ is equivalent to the
condition that there exists a (global) $\Z$ fold
covering $\widetilde{Lag}(M)$ of ${Lag}^+(M)$, which restricts to
the universal covering on each fiber ${Lag}^+(T_pM)$.
(See Lemma 2.6 \cite{Fuk02II}.)
\par
The section $s_L$ lifts to a section $\widetilde s$ of
$\widetilde{Lag}(M)|_L$ if and only if the Maslov class $\mu_L \in H^1(L;\Z)$ of $L$ is zero.
(Recall if $c^1(M) = 0$, then the Maslov class $\mu_L$ is well-defined.)
For each Lagrangian submanifold $L$ with $\mu_L = 0$
a lift $\widetilde s$ of $s_L$ is said to be a {\it grading} of $L$.
The pair $(L,\widetilde s)$ of Lagrangian submanifold $L$ and
grading $\widetilde s$ is called a {\it graded
Lagrangian submanifold}.
\par
Let $(L_i,\widetilde s_i)$ be a graded Lagrangian submanifold.
Then for $p \in L_0 \cap L_1$ we consider any
path $\widetilde{\lambda}$ from $\widetilde s_0(p)$ to $\widetilde s_1(p)$
in $\widetilde{Lag}(T_pM)$ and denote its projection to $Lag(T_pM)$
by $\lambda$. Then we compute the intersection number of
$\lambda$ with the Maslov cycle $Lag_1(T_pM;T_pL_0)$
(relative to $T_pL_0$) to define a degree
$\deg p \in \Z$ for each $p \in L_0 \cap L_1$.
This definition is independent of the choice of $\widetilde{\lambda}$ with
$\widetilde{\lambda}(0) = \widetilde s_0(p)$, $\widetilde{\lambda}(1) = \widetilde s_1(p)$.
(See \cite{seidel:grading}, \cite{Fuk02II} for the details.)
\par
Now we explain how the grading $\lambda$ of $(L,\gamma)$ and
the grading $\widetilde s$ of $L$ are related to each other.
For this purpose, we fix, once and for all, an element $\widetilde V_y$ of $\widetilde{Lag}(T_yM)$
which projects to $V_y \in {Lag}(T_yM)$ at the base point $y$
in Definition \ref{ancgrade}.
\par
First, we go from $\widetilde s$ to $\lambda$.
We consider any anchored Lagrangian submanifold $(L,\gamma)$
with $\mu_L = 0$. Let $\widetilde s$ be a grading of $L$.
We take a section of $\widetilde{\lambda}_{i}$ of the pull-back
$\gamma_{i}^*(\widetilde{Lag}(TM)) \to [0,1]$ such that
$$
\widetilde{\lambda}_{i}(0) = \widetilde V_y, \quad
\widetilde{\lambda}_{i}(1) = s_i(\lambda_i(1)).
$$
Such path is unique up to homotopy because $[0,1]$ is
contractible and so $\gamma_{i}^*\widetilde{Lag}(TM)$ is simply connected.
We push it out and obtain a section $\lambda$ in $\gamma^*Lag(TM)$.
In this way, a graded Lagrangian submanifold $(L,\widetilde s)$
canonically determines a grading $\lambda$ of an anchored Lagrangian submanifold
$(L,\gamma)$. Namely $(\gamma,\lambda)$ becomes a graded anchor
of $L$ in the sense of Definition \ref{ancgrade}.
\par
We remark that the path $\lambda_{01}$ induced by
these graded anchors lifts to $\widetilde{\lambda}_{01}$
joining $s_0(\ell_{01}(0))$ to $s_1(\ell_{01}(1))$.
\par
We then define
$\mu([p,w])$ using this path $\lambda_{01}$ as in Section \ref{sec:grading}.
\begin{lem}\label{degcomp}
$\mu([p,w])$ is independent of $w$. Moreover we have
\begin{equation}\label{eqdegcomp}
\mu([p,w]) = \deg(p).
\end{equation}
\end{lem}
\begin{proof}
Independence of the degree of $w$ is a consequence of our
assumption that Maslov index of $L_0, L_1$ are zero.
Then the equality (\ref{eqdegcomp}) follows easily by
comparing the definitions.
We omit the detail.
\end{proof}
Thus the degree of \cite{Fuk02II} and of this paper coincides under the assumption
that is Maslov index is $0$.
\par
Now we go from $\lambda$ to $\widetilde s$. For any given grading $\lambda$
of $(L,\gamma)$, we lift $\lambda$ to a section of $\gamma^*\widetilde{Lag}(TM)$
so that $\widetilde\lambda(0) = \widetilde V_y$. Then $\widetilde \lambda(1)$
is a lifting of $\lambda(1)$ in $\widetilde{Lag}(T_{\gamma(1)}M$. Since the lifting
of $\widetilde \lambda$ of $\lambda$ is homotopically unique,
$\widetilde \lambda(1)$ depends only on $(L,\gamma)$ and the fixed $\widetilde V_y$.
Therefore if $\mu_L =0$, then this determines a unique grading $\widetilde s$
of $L$ with $\widetilde s(\gamma(1)) =\widetilde \lambda(1)$.
\par
The above discussion can be generalized to the case of $\Z_{2N}$-grading
where the Maslov index is divisible by $2N$ for some positive integer
$N$ rather than being zero. We leave this discussion to the readers.

\section{Reduction of the coefficient ring and Galois symmetry}
\label{sec:Galois}
In this section, we study a reduction of the coefficient ring $\Lambda_{nov}$ to
the subring $\Lambda_{nov}^{\text{\rm rat}}$ or to the ring $\Q[[T^{1/N}]][T^{-1}]$.

\begin{defn}
We put
$$
\Lambda_{nov}^{\text{\rm rat}}
= \left\{ \sum_{i=1}^{\infty} T^{\lambda_i}e^{\mu_i/2}a_i \in \Lambda_{nov}\,
\Big| \, \lambda_i \in \Q \right\}
$$
We also define $\Lambda_{0,nov}^{\text{\rm rat}}$ in a similar way.
\end{defn}
\par
This problem was studied by the first named author in
\cite{fukaya:Galois} in relation to the Galois symmetry of Floer
cohomology over rational symplectic manifolds. Theorem 2.4 in
\cite{fukaya:Galois} is Theorem \ref{GStheorem} of the present
paper. Its proof was given in \cite{fukaya:Galois} as far as $\frak
m_k$ ($k=0,1$) concerns. The case for $k\geq 2$ was `left to the
reader' in \cite{fukaya:Galois}. In this section we give the detail
of the discussion for the case $k\geq 2$.
\subsection{Rational versus BS-rational Lagrangian submanifolds}
In this subsection, we first clarify somewhat confusing usages of
the terminology `rational' Lagrangian submanifolds in the literature
(e.g. in \cite{oh:cyclic}, \cite{fukaya:Galois} etc.).
\par
First we assume that there exists an integer $m_{\text{\rm amb}}$ with $m_{\text{\rm amb}}\omega \in H^2(M;\Z)$,
i.e., $(M,m_{\text{\rm amb}}\omega)$ is integral or pre-quantizable.
Then we choose a complex line bundle $\mathcal P$ with a unitary connection $\nabla$
such that its curvature $F_\nabla$ satisfies
\begin{equation}\label{eq:Fnabla}
F_\nabla = 2\pi \sqrt{-1} m_{\text{\rm amb}}\omega.
\end{equation}
The pair $(\mathcal P,\nabla)$ is called a {\it pre-quantum bundle} of
$(M,m_{\text{\rm amb}}\omega)$.
We note that the connection is flat on any Lagrangian submanifold by
\eqref{eq:Fnabla}.
\begin{defn} We say that a Lagrangian submanifold $L$ is
\emph{Bohr-Sommerfeld $m$-rational} or simply BS \emph{$m$-rational} if
the image of the holonomy group $(\mathcal P|_L,\nabla|_L)$ is contained
in $\{\exp(2\pi\sqrt{-1}km_{\text{amb}}/m) \mid k \in \Z\}$.
We say $L$ is \emph{Bohr-Sommerfeld rational} or simply {\it BS-rational}
if it is BS-rational for some $m$.
We denote the smallest such integer by $m_L$.
\par
When $(\mathcal P,\nabla)$ is trivial on $L$ and $m_{\text{amb}}=1$,
we call $L$ a \emph{Bohr-Sommerfeld orbit}.
\end{defn}
By definition, it is easy to see
$
m_{\text{amb}} \vert m_L
$
for any BS-rational Lagrangian submanifold $L$.
\begin{rem} In \cite{oh:cyclic} and \cite{fukaya:Galois},
the corresponding notion is called \emph{cyclic} and just \emph{rational}
respectively. In this paper, we adopt the name Bohr-Sommerfeld rational
which properly reflects the kind of rationality of
holonomy group of the quantum line bundle $(\mathcal P, \nabla)$
relative to the Lagrangian submanifold.
\end{rem}
A Lagrangian submanifold is often called (spherically) rational in
literature when $\{\omega(\pi_2(M,L))\} \subset \R$ is discrete.
This is related to but not exactly the same as the BS-rationality
in the above definition.
\begin{defn} Let $(M,\omega)$ be rational.
We say that a Lagrangian submanifold is \emph{rational} if
$\Gamma_\omega(L):=\{\omega(\alpha) \mid \alpha \in \pi_2(M,L)\} \subset \R$ is discrete.
\end{defn}
The following lemma shows the relationship between the BS rationality and the rationality.
\begin{lem}\label{lem:cyclic-rational} If $L$ is BS $m$-rational,
then $L$ is rational. Moreover $\Gamma_\omega(L) \subseteq \{\exp(2\pi\sqrt{-1}km_{\text{\rm amb}}/m) \mid k \in \Z\}$.
\end{lem}
The converse does not hold in general when $L$ is not simply connected.
In fact for the case $M = (T^2,dx\wedge dy)$ and $L_t = \{ [t,y] \mid y \in \R\}$,
(Here we regard $T^2 = \R^2/\Z^2$.)
every $L_t$ is rational but only countably many of $L_t$'s
are BS-rational. (The question on which $L_t$ becomes BS-rational depends on the
choice of pre-quantum bundle $(\mathcal P,\nabla)$.
It is easy to see that we may choose the
pre-quantum bundle so that $L_t$ is BS-rational if and only if $t \in \Q$.)
\par
Using Lemma \ref{lem:cyclic-rational}, it is easy to show that the
coefficient ring of the filtered $A_\infty$ algebra $C(L;\Lambda_{0,nov})$ can be
reduced to the ring $\Q[[T^{1/N}]][e,e^{-1}] \subset
\Lambda^{\text{\rm rat}}_{0,nov}$.
In particular, if we define
\begin{equation}
C(L;\Lambda_{0,nov}^{\text{\rm rat}})
= C(L;\Q)\,\, \widehat{\otimes}_{\Q} \,\,\Lambda_{0,nov}^{\text{\rm rat}}
\subset C(L;\Lambda_{0,nov})
\end{equation}
the operations $\frak m_k$ induce a filtered $A_{\infty}$ structure on
$C(L;\Lambda_{0,nov}^{\text{\rm rat}})$.
\subsection{Reduction of the coefficient ring: non-anchored version}
\label{redbynonanchor}
In this subsection, we explain the way to reduce the coefficient ring
of the filtered $A_{\infty}$ category associated
to a symplectic manifold to the ring $\Lambda_{0,nov}^{\text{\rm rat}}$.
\par
Let $(M,\omega)$ be a symplectic manifold with
$m_{\text{\rm amb}}\omega \in H^2(M;\Z)$.
We fix a prequantum bundle $(\mathcal P,\nabla)$ of $(M,m_{\text{\rm amb}}\omega)$.
\par
We fix any integer $N \in \Z_+$
and consider the set of BS $N$-rational Lagrangian submanifolds .
\begin{defn}
The $N$-{\it rationalization} of $L$ (in $(M,\omega)$,
$(\mathcal P,\nabla)$) is a global section $S_L$
of $\mathcal P^{\otimes N/m_{\text{\rm amb}}}$ such that
$$
\Vert S_L\Vert \equiv 1, \quad \nabla^{\otimes N/m_{\text{\rm amb}}} S_L = 0.
$$
\end{defn}
The following lemma is easy to show.
\begin{lem}\label{BSrational}
$L$ is BS $N$-rational if and only if it has an $N$-rationalization.
\end{lem}

Let $L_i$, $i=0, \, 1$ be a pair of $N$ BS-rational Lagrangian submanifolds.
Let $S_{L_i}$ be $N$-rationalizations $L_i$.
For each $p \in L_0 \cap L_1$,
we define $c(p)$ to be the smallest nonnegative real number such that
\begin{equation}\label{defcp}
\exp(2\pi N c(p) \sqrt{-1}/m_{\text{\rm amb}}) S_{L_0}(p) = S_{L_1}(p).
\end{equation}

\begin{prop}\label{prop:E'} Let $\operatorname{length}(\frak L) = k+1 \geq 2$. Define
\begin{equation}\label{eq:E'}
E'(B): = \int_B\omega - \sum_{i=0}^k c(p_{(i+1)i})
\end{equation}
for $B \in \pi_2(\frak L;\vec p)$.
Then $E'$ has values in $\Z[1/N]$ and satisfies the gluing rule
$E'(B) = E'(B_1) + E'(B_2)$ whenever $B = B_1\# B_2$ in the sense of
Lemma \ref{lem:split}.
\end{prop}
\begin{proof}
Let $w\in C^\infty(\dot
D^2,\frak L;\vec p)$ be a map such that $[w] = B$. We consider the pull back bundle
$w^* \mathcal P^{\otimes N/m_{amb}}$.
Let $\gamma : [0,1] \to \partial D^2$ be the map $t \mapsto e^{2\pi\sqrt{-1} t}$.
Using $S_{L_{i}}$ we can construct a section $s$
on $\gamma^*w^* \mathcal P^{\otimes N/m_{\text{\rm amb}}}$ such that
$$\nabla s = 0, \qquad
s(1) = \exp(2\pi N c(p) \sqrt{-1}/m_{\text{\rm amb}}) s(0).
$$
Using the fact that the curvature of $\mathcal P^{\otimes N/m_{\text{\rm amb}}}$ is
$2N\pi\omega$, we conclude $E'(B) \in \Z[1/N]$. The gluing rule for $E'$ is obvious
from its definition \eqref{eq:E'}.
\end{proof}
We put
\beastar
[[ p ]]_{\ell_{01}} & = & T^{- \int w^*\omega + c(p)}[p,w] \\
&{}& \quad \in CF(L_1,L_0;\ell_{01}) \otimes_{\Lambda(L_0,L_1;\ell_{01})} \Lambda_{nov}
\subset C(L_1,L_0;\Lambda_{nov}).
\eeastar
\begin{lem}\label{independetofell01}
If we change the choice of the base point $\ell_{01}$ then there exists
$k \in \Z$ such that
$
[[p]]_{\ell_{01}} = e^k [[p]]_{\ell'_{01}}
$
where $e$ is the formal parameter encoding the degree.
\end{lem}
The proof of the lemma is easy and left to the reader.
\par
This lemma together with the discussion on the degree in the last section shows
that the role of the choice of the base point $\ell_{01}$ is
to fix a connected component of $\Omega(L_0,L_1)$ and does not play
an essential role in Lagrangian Floer theory.
\begin{defn}
We consider the $\Lambda_{nov}^{\text{\rm rat}}$-submodule of $C(L_1,L_0;\Lambda_{nov})$
generated by $[[p]]$ ($p \in L_0 \cap L_1$) and denote it by
$C(L_1,L_0;\Lambda_{nov}^{\text{\rm rat}})$.
We define the module
$C(L_1,L_0,;\Q[[T^{1/N}]][T^{-1}][e,e^{-1}])$ in the same way.
\end{defn}
\begin{rem}
(1)
We note that the $\Lambda_{nov}^{\text{\rm rat}}$-sub-module
$C(L_1,L_0;\Lambda_{nov}^{\text{\rm rat}})$ of
$C(L_1,L_0;\Lambda_{nov})$ depends on the choice of
the rationalizations $S_{L_i}$. We omit them from notation however.
\par
(2)
In Subsection 5.1.3 \cite{fooo08} we put
$$
\langle p\rangle = T^{- \int w^*\omega}[p,w].
$$
Then $C(L_1,L_0;\Lambda_{0,nov})$
is a free $\Lambda_{0,nov}$ module over the basis
$\{ \langle p\rangle \mid p \in L_1 \cap L_0\}.$
\par
The difference of $\langle p\rangle$ from the present basis $[[p]]$ is
$T^{-c(p)}$. Namely
\begin{equation}\label{<>and[[]]}
\langle p\rangle = T^{-c(p)}[[p]]
\end{equation}
\par
$\langle p\rangle$ coincides with the
identity in $
Hom(\mathcal L_p,\mathcal L_p)
$
which appeared in
(2.30) \cite{Fuk02II} and is used there to construct
a filtered $A_{\infty}$ category.
\end{rem}
We now consider the operator
$$
\aligned
&\frak m_{k;m_0,\cdots,m_k}
:
B_{m_k}(C(L_k;\Lambda_{0,nov})) \otimes C(L_k,L_{k-1};\Lambda_{nov})
\otimes \cdots\\
&\quad\otimes C(L_{1},L_0;\Lambda_{nov})
\otimes B_{m_0}(C(L_0;\Lambda_{0,nov}))
\to C(L_{k},L_0;\Lambda_{nov})
\endaligned
$$
\begin{prop}\label{rationality}
The image of
$$
\aligned
&B_{m_k}(C(L_k;\Lambda_{0,nov}^{\text{\rm rat}})) \otimes C(L_k,L_{k-1};\Lambda_{nov}^{\text{\rm rat}})
\otimes \cdots\\
&\quad\otimes C(L_{1},L_0;\Lambda_{0,nov}^{\text{\rm rat}})
\otimes B_{m_0}(C(L_0;\Lambda_{0,nov}^{\text{\rm rat}}))
\endaligned
$$
by $\frak m_{k;m_0,\cdots,m_k}$
is in $C(L_{k},L_0;\Lambda_{nov}^{\text{\rm rat}})$.
\par
The same conclusion holds for $\Q[[T^{1/N}]][T^{-1}][e,e^{-1}]$.
\end{prop}
\begin{proof}
Let $B \in \pi_2(\frak L,\vec p)$
and $\CM(\frak L, \vec p;B)$ be as in Section \ref{chaincomplex}.
\par
For simplicity, we will prove the proposition for the case $m_0 =\cdots=m_k =0$.
Let $\langle p_{ij}\rangle = T^{-\int w_{ij}^*\omega} [p_{ij},w_{ij}]$.
By (\ref{catAinifwob}), we have
\begin{equation}\label{mkcoefficient}
\langle\frak m_k(\langle p_{k(k-1)}\rangle,\cdots,\langle p_{10}\rangle),\langle p_{k0}\rangle\rangle
= \sum_{B \in \pi_2(\mathfrak L,\vec p)} T^{B \cap \omega} e^{\mu(B)/2}
\#\CM(\frak L, \vec p;B).
\end{equation}
Here the left hand side denotes the $\langle p_{k0}\rangle$-coefficient of
$\frak m_k(\langle p_{k(k-1)}\rangle,\cdots,\langle p_{10}\rangle)$.
Therefore by (\ref{<>and[[]]}) we have the matrix coefficients
\begin{equation}\label{mkcoefficient2}
\aligned
&\langle\frak m_k([[p_{k(k-1)}]],\cdots,[[p_{10}]]),[[p_{k0}]]\rangle \\
&= \sum_{B \in \pi_2(\mathfrak L,\vec p)} T^{B \cap \omega - \sum_{i=0}^{k}c(p_{(i+1)i})} e^{\mu(B)/2}
\#\CM(\frak L, \vec p;B)
\endaligned
\end{equation}
for $k \geq 1$.
Since $B \cap \omega - \sum_{i=0}^{k}c(p_{(i+1)i})= E'(B)$ is rational by
Proposition \ref{prop:E'}, the right hand side of (\ref{mkcoefficient2}) lies
in $C(L_{k},L_0;\Lambda_{nov}^{\text{\rm rat}})$ as required.
\end{proof}
Now we are ready to wrap up the proof of Theorem \ref{GStheorem}.
By the assumption $c_1(M) = 0$ and vanishing of Maslov indices of
Lagrangian submanifolds, all Lagrangian submanifolds in the discussion
below carry a grading $\widetilde s$. We just denote $s$ for $\widetilde s$ below
to simplify the notation.
\par
For each given $N$, with $m_{\text{\rm amb}}\vert N$, we construct a filtered $A_{\infty}$ category
over $\Q[[T^{1/N}]][T^{-1}]$.
Its object is $(L, sp, b, s,S_L)$ where $L$ is a BS $N$-rational Lagrangian submanifold
$sp$ its spin structure, $s$ a grading, $b$ is a bounding cochain, and
$S_L$ is $N$-rationalization.
We assume that $b \in C^1(L;\Q[[T^{1/N}]])$.
\par
For two such objects we obtain a $\Q[[T^{1/N}]][T^{-1}]$ module
$$
C(L_1,L_0;\Q[[T^{1/N}]][T^{-1}]).
$$
By Proposition \ref{rationality}, the operation $\frak m_k^{\vec b}$
is defined over this $\Q[[T^{1/N}]][T^{-1}]$.
\par
We have thus obtained a filtered $A_{\infty}$ category
over $\Q[[T^{1/N}]][T^{-1}]$, which we denote by
$$
{\mathcal Fuk}_{N}(M,\omega).
$$
\par
To include all the BS-rational Lagrangian submanifolds and
obtain a filtered $A_{\infty}$ category
over $\Lambda_{nov}^{\text{\rm rat} (0)}$ we proceed as follows.
Let $L_0$, $L_1$ be Lagrangian submanifolds
which are $m_0$-BS rational and $m_1$-BS rational, respectively.
We take there $m_0$ (resp. $m_1$)
rationalization $S_{L_0}$ (resp. $S_{L_0}$).
Take any $N$ such that $ m_0, \, m_1 \mid N$. $S_{L_0}$ (resp. $S_{L_0}$)
induce an $N$ rationalization
Let $S_{L_0}^{N}$ (resp. $S_{L_1}^{N}$) in an obvious way.
(Namely $S_{L_0}^{N} = (S_{L_0})^{ \otimes N/m_0}$.)
\par
For $p \in L_0 \cap L_1$, we use (\ref{defcp}) to obtain $c(p)$.
To make $N$-dependence of $c(p)$ explicit, we write
$c_N(p)$ for $c(p)$. Then for each given $(N,N')$ with
$N | N'$, we have
$
N'c_{N}(p) - c_{N'}(p)=: \Delta(p) \in \Z_{\ge 0}.
$
We put
\begin{equation}
c_{N'}(p) = c_{N}(p) - \frac{\Delta(p)}{N'}.
\end{equation}
We write $[[p]]_{N}$ and $[[p]]_{N'}$ to distinguish the generators of
Floer chain complex over
$\Q[[T^{1/N}]][T^{-1}]$ and over $\Q[[T^{1/N'}]][T^{-1}]$. We consider the map
$$
[[p]]_{N} \mapsto T^{-1/\Delta(p)}[[p]]_{N'},
$$
that induces an isomorphism
$$\aligned
C(L_1,L_0;\Q[[T^{1/N}]][T^{-1}])
&\otimes_{\Q[[T^{1/N}]]} \Q[[T^{1/N'}]][T^{-1}] \\
&\longrightarrow
C(L_1,L_0;\Q[[T^{1/N'}]][T^{-1}])
\endaligned$$
which respect all the $A_{\infty}$ operations.
\par
Therefore the system $({\mathcal Fuk}_{N}(M,\omega); <)$ with
respect to the partial order `$N < N'$ if and only if $N \mid N'$'
forms an inductive system. We define the $A_\infty$-category
${\mathcal Fuk}_{\text{\rm rat}}(M,\omega)$ to be the associated
inductive limit.
\par
Now let $\widehat \Z$ be the profinite completion of $\Z$. As in \cite{fukaya:Galois}, we
will define an action of $\widehat \Z$ on ${\mathcal Fuk}_{\text{\rm rat}}(M,\omega)$.
To define a $\widehat \Z$ action we need to include a flat line bundle $\mathcal L$ over $L$
and take $R = \C$ in place of $R=\Q$.
Namely we take $(L,\mathcal L,sp,b,s,S_L)$
where $(L,sp,b,s,S_L)$ is as before and $\mathcal L$ is a flat $U(1)$ bundle
over $L$.
We say $\mathcal L$ is $N$-rational if
the image of the holonomy representation $\pi_1(L) \to U(1)$ is
contained in $\{ \exp(2\pi\sqrt{-1}k/N) \mid k \in \Z\}$.
\par
Now let $(L_i, \mathcal L_i,sp_i,b_i,s_i,S_{L_i})$ be as above
such that $\mathcal L_i$ are $N$-rational. We put
\begin{equation}
\aligned
C((L_1,\mathcal L_1),&(L_0,\mathcal L_0);\C[[T^{1/N}]][T^{-1}]) \\
&= \bigoplus_{p \in L_0\cap L_1} \Q[[T^{1/N}]][T^{-1}][[p]]
\otimes_{\Q} Hom_{\C}((\mathcal L_1)_p,(\mathcal L_0)_p).
\endaligned\end{equation}
We then modify operations $\frak m_k$ by using the holonomy of $\mathcal L_i$:
Namely we incorporate the holonomy weight in $U(1)$ as defined in (3.28)
\cite{Fuk02II} into the right hand side of (\ref{mkcoefficient2}).
Taking an inductive limit in the same way,
we obtain a filtered $A_{\infty}$ category over $\Lambda_{nov}^{\text{\rm rat}\,(0)\,\C}
=\Lambda_{nov}^{\text{\rm rat}\,(0)} \otimes_{\Q} \C$.
\par
The $m_{\text{\rm amb}}\widehat \Z$ action on it is defined as follows:
Let $m_{\text{\rm amb}} \in m_{\text{\rm amb}}\Z/(N\Z)$ be the standard generator.
We define an action on the set of object
by
$$
m_{\text{\rm amb}} \cdot (L,\mathcal L,sp,b,s,S_L)
= (L,\mathcal L \otimes \mathcal P_L,sp,b,s,S_L)
$$
Since $\mathcal P^{\otimes N/m_\text{\rm amb} }$ is a trivial bundle on $L$,
this defines an action of $m_\text{\rm amb}\Z/N\Z$.
\par
In the same way as \cite{fukaya:Galois} this induces an
action of $m_\text{\rm amb}\Z/N\Z$ on the category over
$\C[[T^{1/N}]][T^{-1}]$.
\begin{rem}
Note the (Galois) action of $1$ on $\C[[T^{1/N}]][T^{-1}]$
is
$$
T^{1/N} \mapsto \exp{(2\pi\sqrt{-1}/N)}T^{1/N}.
$$
This action is consistent with the above action, as was shown in \cite{fukaya:Galois}.
\end{rem}
We then take the inductive limit and obtain an action of $m_\text{\rm amb}\widehat \Z$
on the category over $\Lambda_{nov}^{\text{\rm rat}\,(0)\,\C}$.
The proof of Theorem \ref{GStheorem} is complete.
\qed
\begin{rem}
We may take the maximal abelian extension of $\Q$, that is
the field adding all the roots of unity to $\Q$, in place of $\C$.
\end{rem}
\begin{rem}
We use the coefficient ring $\Lambda_{nov}^{\text{\rm rat}\,(0)\,\C}$
which is the subring of $\Lambda_{nov}^{\text{\rm rat}\,\C}$ consisting of
the series not involving the grading parameter $e$.
This is because we include grading $s$ in the object of our category and so
the Floer cohomology has absolute $\Z$ grading such that
all the operations $\frak m_k$ is of degree $1$ (after shifted).
(We also choose the bounding cochain $b$ so that it is degree $1$.)
\par
We may consider the $\Z_N$-grading instead, then the category is defined over
$\Lambda_{nov}^{\text{\rm rat}\,(0)\,\C}[e^N,e^{-N}]$.
\par
If we take $(L,\mathcal L,sp,b,S_L)$ as an object (that is we do not
include grading at all) then the category is defined over $\Lambda_{nov}^{\text{\rm rat}\,\C}$
\par
We note that the $\widehat{\Z}$ action exists in all of these versions.
\end{rem}
\subsection{The reduction of coefficient ring: anchored version}
\label{subsec:seredanchor}

To see the relation between the
construction of the last subsection to the critical value, it is
useful to consider anchor. Let $y$ be the base point of $M$ we also
fix $V_y \in Lag(T_yM,\omega)$. We take and fix an element $S_y \in
V_y$ such that $\Vert S_y\Vert = 1$. Let $(L_i,\gamma_i)$ be
anchored Lagrangian submanifolds. We assume $L_i$ are $N$-rational.
Then it is easy to see that there exists a unique $N$-rationalization
$S_N$ such that $S_N(\gamma_i(1))$ is a parallel
transport of $S_y$ along $\gamma_i$. Using this rationalization we
discuss in the same way as the last subsection to obtain a non-anchored version:
More specifically, we have the following anchored version of
Proposition \ref{prop:E'}.

\begin{prop}\label{prop:anchoredE'} Let $\CE$ be an anchored Lagrangian chain
of length $\geq 2$. Define a map $E'_k: \pi_2^{ad}(\CE;\vec p) \to \R$
\begin{equation}\label{eq:E'}
E'_k(B): = \int_B\omega - \sum_{i=0}^k c(p_{(i+1)i})
\end{equation}
for $k \geq 2$, and
$$
E'_1(\alpha) = \int_\alpha \omega - c(p_{10})
$$
for $k=1$ and $\alpha \in \pi_2(\ell_{01};p_{01})$.
Then $E'_k$ have their values in $\Z[1/N]$ for all $k = 1, \cdots$ and the collection $E' = \{E'_\ell\}$
defines an abstract index on the collection of
$N$-rational anchored Lagrangian submanifolds.
\end{prop}

The module of morphisms is
$$
CF((L_1,\gamma_1),sp,b_1,\lambda_{1}),((L_0,\gamma_0),sp,b_0,\lambda_{0}))
\otimes_{\Lambda(L_0,L_1;\ell_{01})} \Lambda_{nov}^{\text{\rm rat}}
$$
and higher products are defined in the same way as before.
To define a $\hat\Z$ action on the corresponding $A_\infty$ category,
we include flat $U(1)$ bundles with finite holonomy on $L_i$ as before.
\par
We remark that we obtain the same filtered $A_{\infty}$ category as
the non-anchored version when
$M$ is simply connected. However in general the two are different.

\end{document}